\documentclass[11pt]{amsart}
\usepackage{amsfonts,amssymb,amsxtra}
\usepackage[all]{xy}
\usepackage{eufrak}
\usepackage{latexsym} 

\newcommand{\el}{\par \mbox{} \par \vspace{-0.5\baselineskip}}
\newcommand{\goth}[1]{\EuFrak{#1}}

\newcounter{amoi}
\newtheorem{theo}{Theorem}
\newtheorem{leme}{Lemma}
\newtheorem{prop}{Proposition}
\newtheorem{coro}{Corollary}
\newcommand{\noi}{\noindent}

\newenvironment{rem}{\noi {\el \noi \bf Remark.}}{ \el }
\newenvironment{ex}{\noi {\el \noi \bf Example.}}{ \el }
\newenvironment{exs}{\noi {\el \noi \bf Examples.}}{ \el }
\newenvironment{pf}{\noi {\el \noi \bf Proof.}}{\hfill $\Box$ \el}

\begin{document}
\begin{center}
  {\Large {Bernstein-Gel'fand-Gel'fand reciprocity and indecomposable projective modules for
      classical algebraic supergroups} \vskip 2cm}

Caroline {\sc Gruson} {\footnote{U.M.R. 7502 du CNRS, Institut Elie Cartan, 
Universite Henri Poincar\'e (Nancy 1), BP 239,
54506 Vandoeuvre-les-Nancy Cedex, France. E-mail:
Caroline.Gruson@iecn.u-nancy.fr}}  and Vera {\sc Serganova} {\footnote
{Department of Mathematics, University of California, Berkeley, CA, 94720-3840 USA.  E-mail:
      serganov@Math.Berkeley.EDU}}

\end{center}

\bigskip

{\Small{\bf Abstract:} We prove a BGG type reciprocity law for the category of finite dimensional modules over algebraic supergroups satisfying certain conditions. The equivalent of a standard module in this case is a virtual module called Euler characteristic due to its geometric interpretation. In the orthosymplectic case, we also describe indecomposable projective modules in terms of those Euler characteristics.}

{\Small{\bf Key words:} Finite dimensional representations of algebraic supergroups, Flag variety, BGG reciprocity law.}

\section*{Introduction}

In many representation theories, there exist reciprocity laws. 
Roughly speaking, if the category in question has enough projective
modules, one defines in a natural way a family of so-called
standard modules such that every projective indecomposable module has
a filtration with standard quotients. The reciprocity law states that the multiplicity of a
standard module in  the projective cover of a simple module equals the
multiplicity of this simple module in the standard module. Those
standard modules are usually easy to describe, in particular, their characters
can be described by simple formulae.

For instance, Brauer discovered such a law in the case of finite groups
representations in positive characteristic, \cite{Br}. Another example
is a result of Humphreys, \cite{H} for representations of
semi-simple Lie algebras in positive characteristic. In 1976
(\cite{BGG}) Bernstein,
Gel'fand and Gel'fand introduced the category $\mathcal O$ of highest
weight modules for a semi-simple Lie algebra in characteristic $0$, and
proved a reciprocity law in this category. Irving, \cite{I}, and
Cline, Parshall and Scott, \cite{CPS}, introduced a general notion of highest
weight category and proved a generalized BGG reciprocity. Using this
general approach, it is easy to prove similar results for the
category $\mathcal O$ of highest weight modules for classical
simple Lie superalgebras. For the category of finite-dimensional
representations of classical Lie superalgebras of type I Zou proved
BGG reciprocity in \cite{Z}. For superalgebras of type II the question
remained open, in particular since it was unclear how to define a standard object.  

The first part of this paper (Section 2) is devoted to the generalized
BGG reciprocity for algebraic supergroups $G$ with reductive even part and
symmetric root decomposition. In those cases, the irreducible
representations are parametrized by a highest weight, and if $\lambda$ is a
highest weight, we denote by $L_\lambda$ the corresponding irreducible
representation. Every $L_\lambda$ has an indecomposable projective
cover $P_\lambda$ in the category of finite-dimensional
representations of $G$, \cite{Sqr}. However, in this situation there
is no direct analogue of the so-called standard modules. Hence we
introduce a family of virtual modules $\mathcal E(\mu)$, living in the
Grothendieck group of the category: we call those modules Euler
characteristics because they come from the cohomology of line bundles
on flag supervarieties. It turns out that in the Grothendieck ring the
$[P_\lambda]$-s are linear combinations of $\mathcal E(\mu)$-s and
we denote the coefficient of $\mathcal E(\mu)$ in  $[P_\lambda]$ by $a(\lambda,\mu)$.
The reciprocity law (Theorem \ref{BGG}) states that $a(\lambda,\mu)$ is
exactly the multiplicity of $L_\lambda$ in $\mathcal E(\mu)$. The key argument in the
proof is a $\mathbb Z/2\mathbb Z$-graded analogue of the Bott
reciprocity result, \cite{Bott}, see Proposition \ref{bott}.

All the constructions above depend on the choice of a Borel subgroup in
$G$: in the super case, this choice is not unique up to conjugation,
and the result is true for every possible choice. In particular, in the case of
$GL(m,n)$ our result generalizes Zou's result. It is also to be noted
that the weights $\lambda$ (labeling $L_\lambda$ and $P_\lambda$) and
$\mu$ (labeling $\mathcal E(\mu)$) do not belong to the same set. For
instance, in the orthosymplectic case (Section 4) the $\mu$-s must
have tailless weight diagrams. Finally, let us emphasize on the fact
that this category has infinite cohomological dimension and the
subgroup generated by $[P_\lambda]$-s is a proper subgroup in the
whole Grothendieck group.

The rest of the paper deals with the computation of the coefficients
$a(\lambda,\mu)$ for the orthosymplectic supergroup $SOSP(m,2n)$. The
first computation of those coefficients in the $GL(m,n)$ case was made
in \cite{VSel}. In \cite{B} J. Brundan used another method,
relating this representation theory with the one of $\goth{gl}_\infty$.
 He interpreted the translation functors for $\mathfrak{gl}(m,n)$  as linear operators of  $\mathfrak{gl}_\infty$ acting on 
$\Lambda^n(W)\otimes \Lambda^m(W^*)$, where $W$ is the standard representation of
 $\mathfrak{gl}_\infty$. 
Later on, in \cite{BS1,BS2,BS3} Brundan and Stroppel introduced weight
diagrams, 
which give a clear picture of the translation functors action. Thus
the category of finite dimensional $GL(m,n)$-modules is 
very well understood now, including the projective modules. 

We adopt Brundan's categorification approach. Here we have to separate
in two cases depending on the parity of $m$. If $m$ is odd, the
translation functors can be seen as the Chevalley generators of the Lie
algebra $\goth{gl}_{\infty/2}$ with Dynkin diagram 
$$\circ -\circ - \circ - \ldots \; ,$$
and if $m$ is even, the Lie algebra is  $\goth{gl}_{\infty/2}\oplus \goth{gl}_{\infty/2}$
(see Section 5 and 7).
We compute the coefficients $a(\lambda,\mu)$ in Section 8 (see Theorem \ref{lem1}, Theorem
\ref{prop1} and Theorem \ref{prop1ev}) via
a comparison between the action of translation functors on
$P_\lambda$-s and $\mathcal E(\mu)$-s.  We start with a typical
$\lambda$ (in this case $P_\lambda$, $\mathcal E(\lambda)$ and
$L_\lambda$ coincide) and then obtain an arbitrary $P_\lambda$ by
application of translation functors.

Thus, using the results of this
paper one can express the character of any projective indecomposable
module in terms of $\mathcal E(\mu)$. That however does not imply
automatically an expression of irreducible characters in the same terms. In the case of
$GL(m,n)$ this difficulty can be resolved by allowing infinite linear
combinations of Euler characteristcs $\mathcal E(\mu)$-s, i.e. by completing the
Grothendieck ring. In the case of $SOSP(m,2n)$ the
problem of calculating irreducible characters was solved in \cite{VeraCaroI}  
by calculating Euler characters of vector bundles over an adequate
variety (generalized grassmannian) related to the highest weight and using an induction
on the rank of the supergroup. It seems that this difference between
general linear and orthosymplectic cases is related to the fact that in
the latter case the set of dominant weights has a minimal element with
respect to the standard order.

There remain several open questions such as an interpretation of
indecomposable projectives in terms of canonical bases and the
construction of the analogue of Khovanov's diagram algebra, see
\cite{BS1},\cite{BS2},\cite{BS3} and \cite{BS4}.
It would be also quite interesting to understand how formulae for characters of
the projective modules obtained in this paper are related to the results of
\cite{CLW}.

We thank Jonathan Brundan, Catharina Stroppel and Elizaveta Vishlyakova for fruitful discussions.
This work was partially supported by NSF grant n. 0901554.

\section{Notations and context}

Let $G$ be a connected algebraic supergroup with reductive even part
$G_0$ and $\mathfrak g$ denote its Lie superalgebra. Then $\mathfrak g_0$ is a 
a reductive Lie algebra and $\mathfrak g$ is a semisimple $\mathfrak g_0$-module.
We denote by $\mathfrak h_0$ a Cartan subalgebra of   $\mathfrak g_0$
and by  $\mathfrak h$ a Cartan subalgebra of  $\mathfrak g$, by $H_0$
and $H$ we denote the corresponding algebraic subgroups. By $W$
we denote the Weyl group  $W(\mathfrak g_0, \mathfrak h_0)$.

In order to prove the BGG reciprocity we need the following assumptions
on  $\mathfrak g$ 
\begin{itemize}
\item $\mathfrak h=\mathfrak h_0$, and therefore $H=H_0$;
\item $\mathfrak g_1\simeq \mathfrak g_1^*$ as  a $\mathfrak g_0$-module.
\end{itemize}

Recall that $H=H_0$ is an algebraic torus. Let $\Lambda$ denote the free
abelian group of characters of $H$. One has a root decomposition
$$\mathfrak g=\mathfrak h\oplus\bigoplus_{\alpha\in \Delta}\mathfrak g_\alpha,$$
where
$$\mathfrak g_\alpha=\{x\in\mathfrak g | [h,x]=\alpha(h)x, \forall h \in \mathfrak h\}.$$  
The finite subset $\Delta\subset \Lambda$ is called the set of roots
of $\mathfrak g$. Our assumptions imply that
$\operatorname{dim}\mathfrak g_\alpha=(1,0)$ or $(0,m_\alpha)$ for any root $\alpha\in\Delta$.
In the former case we say that $\alpha$ is even and in the latter that
$\alpha$ is odd. So we have a decomposition
$\Delta=\Delta_0\cup\Delta_1$ defined by the parity of roots. Furthermore,
our assumptions imply that $\Delta=-\Delta$.
It is not difficult to show that one can define a parity function
$p:\Lambda\to\mathbb Z_2$ satisfying the condition
$p(\lambda+\alpha)=p(\lambda)+p(\alpha)$ for all $\lambda\in\Lambda$
and $\alpha\in\Delta$. In general the choice of $p$ is not unique.

As in the case of reductive Lie algebras, we define a decomposition
$\Delta=\Delta^+\cup\Delta^-$ of roots and the corresponding
triangular decomposition 
$$\mathfrak g=\mathfrak n^-\oplus\mathfrak h\oplus\mathfrak n$$
where
$$\mathfrak n^-=\bigoplus_{\alpha\in \Delta^-}\mathfrak g_\alpha,
\mathfrak n=\bigoplus_{\alpha\in \Delta^+}\mathfrak g_\alpha.$$
The subalgebra $\mathfrak b=\mathfrak h\oplus\mathfrak n$ is called a
Borel subalgebra of $\mathfrak g$ and the corresponding algebraic
subgroup $B$ is called a Borel subgroup of $G$. Recall that the
Borel subgroups in $G$ are not always mutually conjugate. It is easy to see that $B$ is the
semi-direct product of the algebraic torus $H$ and the unipotent
supergroup $N$. We set
$$\rho=\frac{1}{2}(\sum_{\alpha\in\Delta_0^+}\alpha-\sum_{\alpha\in\Delta_1^+}m_\alpha
\alpha).$$

By $\mathcal C$ we denote the category of finite-dimensional
$G$-modules, which is isomorphic to the category of 
$\goth g$-modules, semisimple over $\goth h$, with weights in
$\Lambda$, see \cite{Sqr}. It was shown in \cite{Krep} that any simple module in   $\mathcal C$
is a quotient of a Verma module with highest weight $\lambda\in\Lambda$ by a
maximal submodule. A weight $\lambda$ is called {\it dominant} if this
quotient is finite-dimensional. Thus, every dominant weight defines
two simple modules, one is obtained from another by application of the
functor $\Pi$ of change of parity. In order to avoid parity chasing, we
introduce a parity function $p:\Lambda\to\mathbb Z_2$ and  define the category $\mathcal F$ as the full subcategory
of $\mathcal C$ consisting of modules 
such that the parity of any weight  
space coincides with the parity of the corresponding weight. It is not
hard to see that $\mathcal C=\mathcal F\oplus\Pi(\mathcal F)$.

Define the standard order
on $\Lambda$ by setting: $\lambda\leq\mu$ iff
$\mu-\lambda=\sum_{\alpha\in\Delta^+}n_{\alpha}\alpha$ where all
$n_{\alpha}$ are non-negative integers.

For any dominant weight $\lambda$, we will denote by $L_{\lambda}$ the 
simple $\goth g$-module in the category $\mathcal F$ with highest
weight $\lambda$.

It is well-known (see, for example, \cite{Sqr}) that in the category $\mathcal F$ every simple module
$L_\lambda$ has an indecomposable projective cover which we denote by $P_\lambda$.

Denote by $\Lambda^+$ the set of all weights $\lambda$ such that
$\langle\lambda,\check{\beta}\rangle$ is a positive integer for any simple
root $\beta$ of $\Delta^+_0$, where $\check{\beta}\in [\mathfrak  g_\beta,\mathfrak g_{-\beta}]$ such 
that $\langle\beta,\check{\beta}\rangle=2$. 

Let $\mathcal R=\mathbb Z[e^\mu]$ for all $\mu\in\Lambda$.
Let $M\in\mathcal F$ and $m_\mu$ denote the multiplicity of the weight
$\mu$ in $M$ we define the character
$$Ch(M)=\sum_{\mu\in\Lambda}m_\mu e^\mu\in\mathcal R.$$ 
Denote by $\mathcal K(\mathcal F)$ the Grothendieck ring of the
category $\mathcal F$. We will denote by $[M]$ the class of a module $M$ in $\mathcal K(\mathcal F)$ and by $[M:L]$ the multiplicity of an
irreducible module $L$ in the module $M$. 
Clearly, $Ch:\mathcal K(\mathcal F)\to \mathcal R$ is a homomorphism of
rings. Note that
$$Ch(L_\lambda)=e^\lambda+\sum_{\mu <\lambda}m_\mu e^\mu,$$
thus the $Ch(L_\lambda)$ are linearly independent.
Due to our
convention about parity, for any two modules $M$ and $N$ in $\mathcal F$, 
$Ch(M)=Ch(N)$ if and only if $[M]=[N]$ in $\mathcal K(\mathcal F)$.
Hence $Ch$ is injective. For classical Lie superalgebras the image of $Ch$ is described in \cite{SV}.

Let $\goth g$ be a finite-dimensional Kac--Moody superalgebra such
that the quotient of $[\goth g,\goth g]$ by the center is simple. In
other words, $\goth g=\goth{sl}(m,n)$, where $1\leq m<n$,
$\goth{gl}(n,n)$ with $n\geq 2$, $\goth{osp}(m,2n)$, $D(2,1;\alpha)$,
$G_3$ or $F_4$ (see \cite{Kadv}). Let $G$ denote a connected algebraic
supergroup with Lie algebra $\goth g$. For existence of such supergroup
see, for instance, \cite{Sqr}. Then $G$ satisfies our assumptions.

Recall that $\goth g$  is equipped
with a non degenerate invariant bilinear form and the restriction of
this bilinear form to $\goth h$ is also non degenerate. 
Thus, we have a non-degenerate form on $\Lambda$.

Following \cite{Krep} we call a weight $\lambda$ {\it typical} if
$(\lambda+\rho,\alpha)\neq 0$ for any isotropic root $\alpha$. (Recall
that an isotropic root is automatically odd). It follows from
\cite{Krep} that $P_\lambda=L_\lambda$ if and only if $\lambda$ is typical. 

 As follows from \cite{Krep} a typical
$\lambda$ is dominant iff $\lambda+\rho\in\Lambda^+$. For a  general $\lambda$,  the latter 
statement is true only for $\goth{gl}(m,n)$ or $\goth{osp}(2,2n)$ and
a special choice of a Borel subgroup.

Let $\mathcal U (\goth g)$ be the universal enveloping algebra of
$\goth g$ and $\mathcal Z(\goth g)$ be its center. 
For every weight $\lambda$, we write $\chi _{\lambda}$ for the corresponding central character.
A central character $\chi$ is dominant if there exists a dominant
$\lambda$ such that $\chi=\chi _{\lambda}$.

The category $\mathcal F$ splits into direct sum of blocks $\oplus\mathcal F_{\chi}$
consisting of modules admitting the generalized central character $\chi$.
For any $M\in\mathcal F$ we denote by $M_\chi$ the projection of $M$
to the block  $\mathcal F_{\chi}$.

\section{Geometric induction and BGG reciprocity}
In this section we assume that $G$ satisfies the assumptions of Section 1.

Let $B$ be a Borel subgroup of $G$ with Lie algebra $\goth b$, and
let $V$ be a $B$-module. Denote by $\mathcal V$ the induced vector bundle
$G\times_B V$ on the flag supervariety $G/B$. In \cite{VeraCaroI} we
defined
$$\Gamma_i(G/B,V):=H^i(G/B,\mathcal V^*)^*.$$
Recall that $H^i(G/B,\mathcal V^*)$ is a $G$-module, see
\cite{M}, \cite{MPV}, and it is not difficult to see that $\Gamma_i(G/B,V):=H^i(G/B,\mathcal V^*)^*$
is an object of $\mathcal F$ if $V$ satisfies the parity condition
about weights.
 
It is also possible to define $H^i(G/B,\mathcal V^*)$ following
\cite{J} or using the Zuckerman functor approach, see, for instance, \cite{Ssup}. 
Using one of those approaches one can avoid the rather
technical question of existence of $G/B$.

Denote by $C_\lambda$ the one-dimensional $B$-module with weight
$\lambda\in\Lambda$ and by
$\mathcal E(\lambda)$ the class of the {\it Euler characteristic} of the
sheaf $\mathcal C_{\lambda}^*$  belonging to the category $\mathcal F$ , namely
$$\mathcal E(\lambda):= \sum _{\mu} \sum _{i=0} ^{\dim(G/B)_0}(-1)^i[\Gamma_i(G/B, C _{\lambda}): L_{\mu}][L_{\mu}].$$
 
One can easily generalize Proposition 1 in \cite{VeraCaroI} or Theorem
12 in \cite{Ssup} and obtain  the
following character formula for $\mathcal E(\lambda)$
\begin{equation}\label{char}
Ch\; \mathcal E(\lambda)=D \sum_{w\in W}\varepsilon(w) e^{w(\lambda+\rho)},
\end{equation}
where  
$$D_0 = \prod _{\alpha \in \Delta _0 ^+}(e^{\alpha / 2}- e^{-\alpha / 2}),
D_1 = \prod _{\alpha \in \Delta _1 ^+}(e^{\alpha / 2}+ e^{-\alpha / 2})^{m_\alpha},
D=\frac{D_1}{D_0}.$$

Note that (\ref{char}) implies the following 
\begin{leme}\label{symmetry} 
(a) For all $w \in W$, one has $$\mathcal E(\lambda)=\varepsilon(w)\mathcal  E(w(\lambda+\rho)-\rho).$$
In particular, if $\langle\lambda + \rho,\check\beta\rangle=0$ for some even root $\beta$, then $\mathcal E(\lambda)=0$.

(b) The set $$\{Ch\; \mathcal E(\lambda),  \lambda+\rho\in\Lambda^+\}$$ is
linearly independent in $\mathcal R$. 
\end{leme}
\begin{pf} (a) follows immediately from (\ref{char}). To prove (b) we
observe that any $W$-orbit in $\Lambda$ with trivial stabilizer meets
$\Lambda^+$ in exactly one point. Hence  the set $\{\sum_{w\in  W}\varepsilon(w) e^{w(\lambda+\rho)},  \lambda+\rho\in \Lambda^+\}$
is linearly independent in $\mathcal R$. Therefore
(\ref{char}) implies (b).
\end{pf} 

We continue with the following analogue of Bott's reciprocity result.

\begin{prop}\label{bott} Let $\goth n$ denote the maximal nilpotent
ideal of $\goth b$ and $M^{\goth h}$ denote the set of 
$\goth  h$-invariants in an $\goth h$-module $M$. Then, for any
$B$-module $V$ and any dominant weight $\lambda$, we have
$$\left[H^i(G/B,\mathcal V):L_\lambda\right]=\operatorname{dim} H^i(\goth n,
P^*_\lambda\otimes V)^\goth h.$$
\end{prop}

\begin{pf} For every $M\in\mathcal F$ we have
$$[M:L_\lambda]=\operatorname{dim} Hom_\goth g(P_\lambda,M).$$

Consider an injective resolution 
$0\to R_0\to R_1\to \cdots$
of $V$ in the category of $B$-modules. By definition 
$H^i(G/B,\mathcal V)$ is given by the $i$-th cohomology of the
complex
$$0\to H^0(G/B,\mathcal{R}_0)\to  H^0(G/B,\mathcal{R}_1)\to \cdots.$$
Since $ Hom_G(P_\lambda,\cdot)$ is an exact functor, 
$ Hom_G(P_\lambda,H^i(G/B,\mathcal V))$ is given by the
$i$-th cohomology of the complex
$$0\to  Hom_G(P_\lambda,H^0(G/B,\mathcal{R}_0))\to   Hom_G(P_\lambda,H^0(G/B,\mathcal{R}_1))\to \cdots.$$
Using Frobenius reciprocity we have 
$$Hom_G(P_\lambda,H^0(G/B,\mathcal{R}_j))=Hom_B(P_\lambda,R_j).$$
Thus, we obtain that 
$$Hom_G(P_\lambda,H^i(G/B,\mathcal V))=Ext^i_B(P_\lambda,V).$$

Now we recall that $B$ is the semidirect product of the torus $H$ and
the unipotent supergroup $N$. Therefore it is easy to see that
$$Ext^i_B(P_\lambda,V)=Ext^i_\goth n((P_\lambda,V)^\goth h=H^i(\goth n,P^*_\lambda\otimes V)^\goth h.$$
\end{pf}

By dualizing we obtain the following corollary.

\begin{coro}\label{kostant} Let $\goth n$ denote the maximal nilpotent
ideal of $\goth b$ and $M^{\goth h}$ denote the set of 
$\goth  h$-invariants in an $\goth h$-module $M$. Then, for any weight
$\nu$ and any dominant weight $\lambda$, we have
$$\left[\Gamma_i(G/B,C_\nu):L_\lambda\right]=\operatorname{dim} H^i(\goth n,
P_\lambda\otimes C_{-\nu})^\goth h.$$
\end{coro}

Denote by $b(\nu,\lambda)$ the coefficients in the decomposition
$$\mathcal E(\nu)=\sum_\lambda b(\nu,\lambda) [L_\lambda].$$
Lemma \ref{symmetry} implies
\begin{equation}\label{coeff} 
b(\nu,\lambda)=\varepsilon(w)b(w(\nu+\rho)-\rho,\lambda).
\end{equation}

\begin{coro}\label{kostcor} The coefficient $b(\nu,\lambda)$ is equal to
  the constant term in the formal expression
$D^{-1}e^{-\nu-\rho}Ch(P_\lambda)$.
\end{coro}
\begin{pf} - By proposition \ref{kostant}
$$b(\nu,\lambda)=\sum_i 
(-1)^i \operatorname{dim}H^i(\goth n,P_\lambda\otimes C_{-\nu})^\goth h.$$
Using the complex which computes the Lie superalgebra
cohomology we obtain 
$$\sum_i (-1)^i\;Ch\;H^i(\goth n,P_\lambda\otimes C_{-\nu})=
\sum_i (-1)^i\;Ch(P_\lambda\otimes C_{-\nu}\otimes \Lambda^i(\goth n^*))=$$
$$=e^{-\nu}\;Ch\;P_\lambda\sum_i (-1)^i \;Ch\;\Lambda^i(\goth
n^*).$$
To finish the proof, use 
$$\sum_i (-1)^i \;Ch\;\Lambda^i(\goth n^*)=
\frac{\prod_{\alpha\in\Delta_0^+}(1-e^{-\alpha})}{\prod_{\alpha\in\Delta_1^+}(1+e^{-\alpha})^{m_\alpha}}=D^{-1}e^{-\rho}.$$ 
\end{pf}

\begin{coro}\label{charcor} - One has:
$$Ch\;P_\lambda=D\sum_{\nu+\rho\in \Lambda^+} b(\nu,\lambda)\sum_{w\in W}\varepsilon(w)e^{w(\nu+\rho)}.$$
\end{coro}
\begin{pf} - The statement follows from Corollary \ref{kostcor} and
  (\ref{coeff}) since  Corollary \ref{kostcor} implies
$$Ch\;P_\lambda=D\sum_{\nu\in \Lambda} b(\nu,\lambda)e^{\nu}.$$
\end{pf}

\begin{theo}\label{BGG} We have the following identity in
  $\mathcal K(\mathcal F)$
$$[P_\lambda]=\sum_{\mu+\rho\in\Lambda^+}a(\lambda,\mu)\mathcal
  E(\mu).$$
Moreover, the following analogue of BGG reciprocity holds
$$a(\lambda,\mu)=b(\mu,\lambda).$$
\end{theo}

\begin{pf} - The statement follows from Corollary \ref{charcor} and (\ref{char}). 
\end{pf}

\begin{ex} Let $G=GL(m,n)$ and $B$ be the subgroup of upper triangular
  matrices. Then $\Lambda^+-\rho$ coincides with the set of dominant weights.
Moreover, it is well-known (see for example \cite{P}) that for any $\lambda\in\Lambda^+$, 
$\Gamma_i(G/B,C_\lambda)=0$ if $i>0$. Moreover,
$$\Gamma_0(G/B,C_\lambda)\simeq K_\lambda:=U(\goth g)\otimes_{U(\goth g^+)}L_\lambda^0,$$
where $\goth g^+=\goth g_0+\goth{b}$ and $L_\lambda^0$ is the
irreducible $\goth g_0$-module of highest weight $\lambda$ with trivial action of $\goth b_1$. The
module $K_\lambda$ was first considered in \cite{Krep} and is usually 
called a Kac module. It was proven in \cite{Z} that every indecomposable
projective
module $P_\lambda$ has a filtration by Kac modules $K_\mu$ and that the
multiplicity of $K_\mu$ in $P_\lambda$ equals the multiplicity of  $L_\lambda$ 
in  $K_\mu$. A combinatorial algorithm for calculating
$a(\lambda,\mu)$ in this case was obtained by Brundan, \cite{B}. We
will explain it in Section \ref{wd} after introducing weight diagrams. 
\end{ex}

\section{Classical supergroups $GL(m,n)$ and $SOSP(m,2n)$}

In this section we collect all necessary facts about roots and weights 
for the classical supergroups. So we assume that $G=GL(m,n)$ or $SOSP(m,2n)$.

The lattice $\Lambda$ of all integral weights is 
$$\Lambda=\bigoplus_{i=1}^m\mathbb Z\varepsilon_i\oplus\bigoplus_{j=1}^n\mathbb Z\delta_i.$$
We define a parity homomorphism $p:\Lambda\to \mathbb Z_2$ by
$p(\varepsilon_i)=0, p(\delta_j)=1$ for all $1\leq i\leq m$, $1\leq j\leq n$. 
The invariant form on $\Lambda$ is given by
$$ (\varepsilon_i,\varepsilon_j)=\delta_{ij},\,\,(\varepsilon_i,\delta_j)=0,\,\,(\delta_i,\delta_j)=-\delta_{ij}.$$

Now we recall the description of $\Delta$ (see ~\cite{Kadv}). 

The even roots $\Delta_0$ of $\goth {gl}(m,n)$ are all the vectors of the
form $\varepsilon_i-\varepsilon_j$ and $\delta_i-\delta_j$ with $i\neq
j$. The odd roots $\Delta_1$ of $\goth {gl}(m,n)$ are all the vectors of
the form  $\varepsilon_i-\delta_j$ and $\delta_i-\varepsilon_j$.

The even roots $\Delta_0$ of $\goth {osp}(2m,2n)$ are all the vectors of the
form $\pm\varepsilon_i\pm\varepsilon_j$, $\pm\delta_i\pm\delta_j$
(the signs can be chosen independently) with $i\neq j$ and $2\delta_i$. The odd roots $\Delta_1$ of 
$\goth {osp}(2m,2n)$ are all the vectors of
the form  $\pm \varepsilon_i\pm \delta_j$. 

The even roots $\Delta_0$ of $\goth {osp}(2m+1,2n)$ are all the vectors of the
form $\pm\varepsilon_i\pm\varepsilon_j$, $\pm\delta_i\pm\delta_j$
with $i\neq j$, $\pm\varepsilon_i$ and $\pm 2\delta_i$. The odd roots $\Delta_1$ of 
$\goth {osp}(2m+1,2n)$ are all the vectors of
the form  $\pm \varepsilon_i \pm\delta_j$ and $\pm\delta_i$.

From now on we fix a Borel subalgebra of $\mathfrak g$, we make the same
choice as in \cite{VeraCaroI}.
Below is the list of the simple roots for our choice of Borel
subalgebras. 

$\bullet$ If $\goth g=\goth {gl}(m,n)$, $m\geq n$, the simple roots are
$$\varepsilon_1-\varepsilon_2,\varepsilon_2-\varepsilon_3,...,\varepsilon_m-\delta_1,\delta_1-\delta_{2},...,\delta_{n-1}-\delta_n,$$
$$\rho=\frac{m-n-2}{2}\varepsilon_1+\frac{m-n-4}{2}\varepsilon_2+...+\frac{-m-n}{2}\varepsilon_m+
\frac{m+n}{2}\delta_1+...+\frac{m-n+2}{2}\delta_n;$$

$\bullet$ If $\goth g=\goth {osp}(2m+1,2n)$ and $m\geq n$, the simple
roots are
$$\varepsilon_1-\varepsilon_2,...,\varepsilon_{m-n+1}-\delta_1,\delta_1-\varepsilon_{m-n+2},...,\varepsilon_{m}-\delta_n,\delta_n,$$
$$\rho=-\frac{1}{2}\sum_{i=1}^m \varepsilon_i+\frac{1}{2}\sum_{j+1}^n\delta_j+\sum_{i=1}^{m-n}(m-n-i+1)\varepsilon_i;$$
 
$\bullet$ If $\goth g=\goth {osp}(2m+1,2n)$ and $m<n$, the simple
roots are
$$\delta_1-\delta_2,...,\delta_{n-m}-\varepsilon_1,\varepsilon_1-\delta_{n-m+1},...,\varepsilon_{m}-\delta_n,\delta_n,$$
$$\rho=-\frac{1}{2}\sum_{i=1}^m
\varepsilon_i+\frac{1}{2}\sum_{j+1}^n\delta_j+\sum_{j=1}^{n-m}(n-m-j)\delta_j;$$

$\bullet$ If $\goth g=\goth {osp}(2m,2n)$ and $m>n$, the simple
roots are
$$\varepsilon_1-\varepsilon_2,...,\varepsilon_{m-n}-\delta_1,\delta_1-\varepsilon_{m-n+1},...,\delta_n-\varepsilon_m,\delta_n+\varepsilon_m,$$
$$\rho=\sum_{i=1}^{m-n}(m-n-i)\varepsilon_i;$$

$\bullet$ If $\goth g=\goth {osp}(2m,2n)$ and $m\leq n$, the simple
roots are
$$\delta_1-\delta_2,...,\delta_{n-m+1}-\varepsilon_1,\varepsilon_1-\delta_{n-m+2},...,\delta_n-\varepsilon_m,\delta_n+\varepsilon_m,$$
$$\rho=\sum_{i=1}^{n-m}(n-m-i+1)\delta_i.$$

Finally, we give a description of $\Lambda^+$.
Let 
$$\lambda+\rho=a_1\varepsilon_1+...+a_m\varepsilon_m+b_1\delta_1+...+b_n\delta_n.$$
Then $\lambda$ is integral iff $a_i,b_j\in\mathbb Z$ for $G=GL(m,n)$
or $SOSP(2m,2n)$, and 
$a_i,b_j\in\frac{1}{2}+\mathbb Z$ for $G=SOSP(2m+1,2n)$.
Furthermore, $\lambda+\rho\in \Lambda^+$ if  

$$a_1>a_2>...>a_m \;,\; b_1>b_2>...>b_n\;\;\text{if}\;\; G=GL(m,n);$$

$$a_1>a_2>...>a_{m}\geq \frac{1}{2} \; , \;b_1>b_2>...>b_{n}\geq\frac{1}{2}\;\;\text{if}\;\; G=SOSP(2m+1,2n);$$

$$a_1>a_2>...>a_{m-1} > |a_{m}| \; and  \; b_1>b_2>...>b_{n}>0,\;\;\text{if}\;\; G=SOSP(2m,2n).$$

Every $\lambda\in\Lambda^+-\rho$ is dominant. If $G=GL(m,n)$, the set of
dominant weights coincides with $\Lambda^+-\rho$. In the orthosymplectic case
we will formulate the condition of dominance in the next section. 

\section{Weight diagrams}\label{wd}

We recall the definitions and notations for weight diagrams,
introduced in \cite{BS1} for $GL(m,n)$, and in 
\cite{VeraCaroI} for $SOSP(m,2n)$. Note that our notations slightly differ
from those in  \cite{BS1}, for translation see \cite{MS}.

Let $\mathbb T\subset \mathbb R$ be a discrete set,
$X=(x_1,...,x_m)\in
\mathbb T^m$,
$Y=(y_1,...,y_n)\in\mathbb T^n$. A diagram 
$f_{X,Y}$ is a function defined on $\mathbb T$ whose values are 
multisets with elements $<,>,\times$ according to the following algorithm.

$\bullet$ Put the symbol $>$ in position  $t$ for all $i$ such that $x_i=t$.

$\bullet$ Put the symbol $<$ in position $t$ for all $i$ such that $y_i=t$.

$\bullet$ If there are both $>$ and $<$ in the same position replace them by 
the symbol $\times $, repeat if possible.

Thus, $f_{X,Y}(t)$ may contain at most one of the two symbols $>,<$.
We represent $f_{X,Y}$ by the picture with $\circ$ standing in position 
$t$ whenever $f(t)$ is an empty set.

Let  $G=GL(m,n)$. Let $\lambda$ be a dominant integral weight such that
$$\lambda+\rho=a_1\varepsilon_1+...+a_m\varepsilon_m+b_1\delta_1+...+b_n\delta_n.$$

Set $\mathbb T=\mathbb Z$, 
$$X_{\lambda}=(a_1,...,a_m), Y_{\lambda}=(-b_1,...,-b_n).$$ 
The diagram $f_{\lambda}=f_{X_{\lambda},Y_{\lambda}}$ is called the
{\it weight diagram} of $\lambda$.

A diagram is the weight diagram of some dominant weight if and only if
$f(t)$ is empty or is a single element set since both sequences
$a_1,...,a_m$ and $b_1,...,b_m$ are strictly decreasing and hence do
not have repetitions.

Now let $G=SOSP(2m,2n)$. 
Set $\mathbb T=\mathbb Z_{\geq  0}$.
For a dominant weight $\lambda$ such that 
$\lambda+\rho=a_1\varepsilon_1+...+a_m\varepsilon_m+b_1\delta_1...+b_n\delta_n$
let 
$$X_{\lambda}=(a_1,...,a_{m-1},|a_m|), Y_{\lambda}=(b_1,...,b_n), 
f_{\lambda}=f_{X_{\lambda},Y_{\lambda}}.$$

The position $0$ is called the {\t tail position}. If the tail position is
empty we put $[+]$ or $[-]$ before the diagram if $a_m>0$ or $a_m<0$ respectively.

A diagram $f_{\lambda}$ is the weight diagram of a
dominant $\lambda$ if and only if

$\bullet$ for any $t\neq 0$, $f_{\lambda}(t)$ is empty or a single element set;

$\bullet$ the multiset $f_{\lambda}(0)$ does not contain $<$, contains $>$ with
multiplicity at most 1 (it may contain any number of $\times$).

A diagram $f_\lambda$ and a weight $\lambda$ are called {\it tailless} if $f(0)$ does not contain
$\times$. A weight $\lambda$ is tailless iff $\lambda+\rho\in\Lambda^+$.

Finally, let  $G=SOSP(2m+1,2n)$.
 Let 
$\mathbb T=\frac{1}{2}+\mathbb Z_{\geq 0}$ and define
$X_{\lambda}$, $Y_{\lambda}$ and $f_{\lambda}$ as in the case
$\goth g=\goth{osp}(2m,2n)$.
The dominance condition is equivalent to the following condition on a
weight diagram $f$

$\bullet$ $f(t)$ is empty or a single element set for any $t\neq \frac{1}{2}$;

$\bullet$ $f(\frac{1}{2})$ may contain at most one of $<$ or  $>$ and any
number of $\times$. 

The position $\frac{1}{2}$ is called the {\it tail position}.
It is possible that two dominant weights have
the same weight diagram. That may happen if $f(\frac{1}{2})$ does not contain
$>$ or $<$ and has at least one $\times$. For example,
the diagram with two $\times$ at $\frac{1}{2}$ corresponds to
$(\frac{1}{2},-\frac{1}{2}|\frac{1}{2},\frac{1}{2})$ and to
$(-\frac{1}{2},-\frac{1}{2}|\frac{1}{2},\frac{1}{2})$. 
So if the weight diagram has at least one $\times $
and no $<,>$ at the position $\frac{1}{2}$ we put an indicator (which we sometimes refer to as "sign") $(\pm)$ before the weight
diagram. Its value is $+$ if the corresponding weight has the form 
$$\lambda+\rho=(a_1,...,a_{m-s},\frac{1}{2},-\frac{1}{2},...,-\frac{1}{2}|b_1,...,b_{n-s},\frac{1}{2},...,\frac{1}{2}),$$
and $-$ if the corresponding weight has the form
$$\lambda+\rho=(a_1,...,a_{m-s},-\frac{1}{2},-\frac{1}{2},...,-\frac{1}{2}|b_1,...,b_{n-s},\frac{1}{2},...,\frac{1}{2}),$$
where $s$ is the number of crosses at the position $\frac{1}{2}$.

A weight $\lambda$ is {\it tailless} if $f_\lambda (\frac{1}{2})$ has
at most one symbol, and if this symbol is $\times$ the
indicator is $+$. As in the case $G=SOSP(2m,2n)$ a weight is tailess
iff $\lambda+\rho\in\Lambda^+$.  

Thus, tailless diagrams are exactly the diagrams of weights in $\Lambda^+-\rho$.

Recall that the number of $\times$ in the diagram is called 
the degree of atypicality of the corresponding weight. 
The symbols $>$, $<$ are called the {\it core symbols} of the
diagram. The {\it core} of a weight $\lambda$ is the diagram
$f_\lambda$ with all $\times$ and the sign indicator removed except
the case  when $G=SOSP(2m,2n)$ and there are no $\times$. 
As follows from \cite{VeraCaroI} two weights admit the same central character (i.e. the
corresponding simple modules belong to the same block) if the cores of
the two weights are the same.
Thus, if two weights have the same central
character, then they have the same degree of atypicality. Hence the degree of
atypicality of a central character $\chi$ is well defined. We denote
it by $at(\chi)$. A weight ( resp. central character) is typical if its
degree of atypicality is zero. If $\chi$ is typical, then $\mathcal F_\chi$ is semisimple and has
only one simple object up to isomorphism.

Let $\lambda$ be dominant and tailless. Following Brundan and Stroppel  
\cite{BS4}
we define the cap diagram of $f_\lambda$ by the following rules. The
left end of a cap is at a $\times$ and the right end is at an empty
position. We start          
from the rightmost $\times$, make the cap by             
joining it to the next free position on the right  (the end is not free any more), and then repeat for the        
next $\times$ to the left, and so on until there is no $\times$ left.  There is no empty position under any cap. One can see
that for each weight diagram there is a unique cap diagram. 
\begin{ex} For the following weight diagram,

\xymatrix{
&&&&&&\\
{\circ} & {\times} & > & < & \times& \circ & \circ }

~\\
the caps are the following: 
 
 \xymatrix{
 & & & & & & \\
{\circ} & {\times}\ar@{-} `u[ur] `[rrrrr]  [rrrrr] & > & < & \times\ar@{-} `u[r] `[r]  [r] & \circ & \circ . }

~\\

 \end{ex}

Denote by $\mathcal P (\lambda)$ the set of all the weights $\mu$ with
weight diagram $f_{\mu}$ which can be obtained from $f_{\lambda}$  by 
moving some of the $\times$-s from the left end to the right end of a
cap. Note that the cardinality of  $\mathcal P (\lambda)$ is $2^k$,
where $k$ is the total number of $\times$.

If $G=GL(m,n)$, then all diagrams are tailless.
The following result was proven in \cite{B}.

\begin{prop}\label{br} - Let $G=GL(m,n)$. Then
$$[P(\lambda)] = \sum _{\mu \in \mathcal P (\lambda)}  \mathcal E  (\mu).$$ 

\end{prop}

\section{Categorification of $\mathfrak{gl}_{\infty /2}$ in orthosymplectic terms}\label{cat}

In this section we assume that $G=SOSP(2m+1,2n)$.
We denote by $\mathcal K(\mathcal E)$ the subgroup of $\mathcal K(\mathcal F)$ (resp. $\mathcal K(P)$)  
generated by the Euler characteristics $\mathcal E(\lambda)$ for all
$\lambda+\rho\in\Lambda^+$ (resp. by the classes of indecomposable projective
modules $P_\lambda$ for all dominant $\lambda$).

\subsection {Representation of  $\mathfrak{gl}_{\infty /2}$}
We denote by $\mathfrak{gl}_{\infty /2}$ the infinite dimensional Lie
algebra defined over $\mathbb Z$ with Dynkin diagram
$$\circ - \circ - \circ - \circ  \ldots ,$$ and by $V$ its standard
representation with basis 
$v_{1/2}, v_{3/2}, \ldots$. We note $w_{1/2}, w_{3/2}, \ldots$ the
basis in $V^*$ such that $\langle w_i,v_j\rangle=\delta_{ij}(-1)^{i+1/2}$. Let $E_{i,j}$ be the element of 
$\mathfrak{gl}_{\infty /2}$ which acts on $V$ the following way:
$E_{i,j}(v_k) = \delta _{jk}v_i$ and on 
$V^*$: $E_{i,j}(w_k) = (-1)^{i+j}\delta _{ik}w_j $.

Consider the representation (defined in the natural way) $\Lambda ^m (V^*)\otimes \Lambda ^n (V)$ of 
$\mathfrak{gl}_{\infty /2}$. Then, if $\lambda$ is a tailless dominant weight for  $\mathfrak{osp}(2m+1,2n)$, 
such that $\lambda + \rho = \sum _{i=1} ^m a_i \varepsilon _i + \sum _{j=1}  ^n b_j \delta _j$ 
with $a_i,b_j \in 1/2 + \mathbb Z$, we associate to $\lambda$ the following vector in 
$\Lambda ^m (V^*)\otimes \Lambda ^n (V)$: 
$$x_{\lambda}:= w_{a_1}\wedge \ldots \wedge w_{a_m} \otimes v_{b_1}\wedge \ldots \wedge v_{b_n}.$$  
We denote the weight of the action of $\mathfrak{gl}_{\infty /2}$ on $v_i$ by $\gamma _i$ (hence, 
the weight of the action on $w_i$ is $-\gamma _i$), so that every vector $x_{\lambda}$ is equipped with 
a $\mathfrak{gl}_{\infty /2}$-weight which we denote by $\gamma (\lambda)$.

\begin{rem} - The weights with the same core have the same $\mathfrak{gl}_{\infty /2}$-weight. Therefore 
for every dominant weights $\lambda, \mu$, one has
$$\chi _{\lambda} = \chi _{\mu} \Leftrightarrow \gamma (\lambda) = \gamma (\mu).$$
\end{rem}

Define
the map $\varphi: \mathcal K(\mathcal E)\to \Lambda ^m (V^*)\otimes\Lambda ^n (V)$ by
$\varphi(\mathcal E(\lambda))=x_\lambda$. Obviously, $\varphi$ is an
isomorphism of abelian groups.

\subsection {Interpretation of the translation functors in this setting}

Since the weight of the $\mathfrak{gl}_{\infty /2}$-action encodes the central
character, every block in $\mathcal F$ can be parametrised by a 
$\mathfrak{gl}_{\infty /2}$-weight $\gamma$, hence $\mathcal F = 
\oplus _{\gamma} \mathcal F_{\gamma}$. The translation functors consist 
in tensoring with the standard module $E$ of $\mathfrak{osp}(2m+1,2n)$ 
and then projecting in the appropriate block.

For any $M \in \mathcal F _{\gamma}$, we set:
$$T_{a,a+1}(M) = (M\otimes E)_{\gamma + \gamma _a - \gamma _{a+1}},$$
$$T_{a+1,a}(M) = (M\otimes E)_{\gamma + \gamma _{a+1} - \gamma _{a}}.$$
The functors $T_{a,a+1}$ and $T_{a+1,a}$, being exact, induce linear 
operators in the Grothendieck group $\mathcal K (\mathcal F)$, and we keep 
the same notations for them. Note also that $T_{a,a+1}$ and
$T_{a+1,a}$ are adjoint, i.e.
\begin{equation}\label{adjoint}
Hom_\goth g(M,T_{a,a+1}(N))=Hom_\goth g(T_{a+1,a}(M),N)
\end{equation}

Since, as can be read in \cite{VeraCaroI} (Lemma 1(b)and Corollary 1) we have
$$\Gamma_i(G/B,C_{\lambda}\otimes E)= \Gamma_i(G/B,C_{\lambda})\otimes E,$$
we see that $\mathcal K(\mathcal E)$ 
is invariant under both $T_{a,a+1}$ and $T_{a+1,a}$.
\begin{leme} \label{Lemmaa} - For any $a\geq \frac{1}{2}$ one has 
$$\varphi \circ T_{a+1,a} = E_{a+1,a}\circ \varphi$$
and
$$\varphi \circ T_{a,a+1} = E_{a,a+1}\circ \varphi .$$ 
\end{leme}
\begin{pf} - Let $\gamma '$ be a $\mathfrak{gl} _{\infty /2}$-weight and let
$\lambda$ be a tailless dominant $\mathfrak{osp}(2m+1,2n)$-weight with
$\gamma(\lambda)=\gamma$.
Corollary 1 in \cite{VeraCaroI} implies that 
\begin{equation}\label{tens}
\Gamma _i (G/B,(C_{\lambda}\otimes E)_{\Phi ^{-1}(\gamma ')})
= (\Gamma _i(G/B,C_{\lambda})\otimes E)_{\gamma '},
\end{equation}
where ${\Phi ^{-1}(\gamma ')}$ is the set of all weights $\mu$ such
that $\gamma(\mu)=\gamma'$.

The $\mathfrak b$-module $L_{\lambda}(\mathfrak b)\otimes E$ has a filtration 
by all the $C_{\mu}$ with $\mu = \lambda \pm \varepsilon _j$, 
$\lambda \pm \delta _k$, $\lambda$. Since the Euler characteristic 
of the associated graded sheaf coincides with the Euler characteristic
of the original sheaf, we obtain
$$T_{a,a+1}(\mathcal E(\lambda))=\sum_{\mu}\mathcal E(\mu),$$
where the sum is taken over all 
$\mu=\lambda\pm\varepsilon_j,\lambda\pm\delta_k$
such that $\gamma(\mu)=\gamma(\lambda)+\gamma_a-\gamma_{a+1}$.

By direct inspection if $\mu$ is not dominant, 
then $(\mu+\rho,\beta)=0$ for some even root $\beta$ and therefore by
(\ref{char}) $\mathcal E(\mu)=0$. Hence in fact only dominant $\mu$
contribute in the summation. To find all such $\mu$
we use the weight diagram 
$f_{\lambda}$ of $\lambda$, and we only have to look at the positions 
$a$ and $a+1$ in $f_{\lambda}$. Here is a table of the different 
$f_{\mu}$ which can occur (see figure 1). Note that since only tailless weights appear, if the left position is $\frac{1}{2}$,  and there is a $\times$ then the sign before it should be $(+)$ and we omit it for simplicity of notation.

\bigskip

$$\begin{array}{ccccccccc}
f_{\lambda} & a & a+1 & f_{\mu} & a & a+1 &\\
&&&&&\\
& \circ & \circ && \emptyset &&&&\\
&\circ & < && <&\circ&&&\\
& \circ & >&&\emptyset &&&&\\
& \circ & \times && <&>&&&\\
& < & \circ && \emptyset &&&&\\
& < & < && \emptyset &&&&\\
& < & > && \emptyset &&&&\\
& <  & \times && \emptyset& &&&\\
& > & \circ && \circ&>&&&\\
& > & < && \circ\;&\times&\oplus &\times \;& \circ\\
& > & > && \emptyset&&&&\\
& > & \times && \times & >&&&\\
& \times & \circ && <&>&&&\\
& \times & < && <&\times&&&\\
& \times & > && \emptyset&&&&\\
& \times & \times && \emptyset &&&&\\
&&(figure \; 1)&&&&
\end{array}$$
 
Now we conclude, noticing that $$T_{a,a+1}(\mathcal E _{\lambda}) = \sum _{\mu}
 \mathcal E _{\mu}$$ with the weights $\mu$ we computed, and since
$$E_{a,a+1}(x_{\lambda}) = \sum _{\mu} x_{\mu}$$
by construction, we have proved the first identity. The proof of the
second identity is similar and we leave it to the reader.
\end{pf}
\begin{rem} -  The operators $E_{a,a+1}$ and $E_{a+1,a}$ satisfy the Serre 
relations.
\end{rem}

\subsection{The switch functor} 
Recall that the switch functor $sw:\mathcal F_\gamma\to\mathcal F_\gamma$ is
defined by
$$sw(M)=(M\otimes E)_\gamma.$$
We  also denote by $sw$ the corresponding linear operator in $\mathcal K(\mathcal F)$.

The following proposition explains how $sw$ acts on $\mathcal E(\lambda)$.

\begin{leme}\label{propalpha} - Let $\lambda$ be a tailless dominant
  weight. 
If $f_\lambda$ has an empty tail position, then
$sw(\mathcal E(\lambda))=\mathcal E(\lambda)$. 

If $f_\lambda$ has $\times$ at the tail position, then
$sw(\mathcal E(\lambda))=-\mathcal E(\lambda)$. 

If $f_\lambda$ has $>$ or $<$ at the tail position, then
$sw(\mathcal E(\lambda))=0$. 
\end{leme}
\begin{pf} - We use the same idea as in the proof of Lemma \ref{Lemmaa}.
As follows from Corollary 1 in \cite{VeraCaroI}
$$sw(\mathcal E(\lambda))=\sum_{\mu} \mathcal E(\mu),$$
where the summation is over all 
$\mu=\lambda\pm\varepsilon_j,\lambda,\lambda\pm\delta_k$
such that $\gamma(\mu)=\gamma(\lambda)$.

If $f_\lambda$ has empty tail position, then the only weight
$\mu$ appearing in the sum is $\lambda$.

If  $f_\lambda$ has $>$ at the tail position, then the weights $\mu$ appearing in
sum are $\lambda$ and
$\lambda-\varepsilon_m$. But $\lambda-\varepsilon_m=s(\lambda+\rho)-\rho$ where
$s$ is the reflection with respect to the root $\varepsilon_m$. Hence
by (\ref{char})
$\mathcal E(\lambda-\varepsilon_m)=-\mathcal E(\lambda)$
and $sw(\mathcal E(\lambda))=0$. 

Similarly, if $f_\lambda$ has $<$ at the tail position,
$$sw(\mathcal E(\lambda))=\mathcal E(\lambda)+\mathcal E(\lambda-\delta_n)=0.$$

Finally, if  $f_\lambda$ has $\times$ at the tail position, then
$$sw(\mathcal E(\lambda))=\mathcal E(\lambda)+\mathcal E(\lambda-\varepsilon_m)+\mathcal E(\lambda-\delta_n)=-\mathcal E(\lambda).$$
\end{pf}

\section {Translation functors action on simple and projective modules}

In this section $G=SOSP(2m+1,2n)$.

\subsection{Functors $T_{a,a+1}$ and $T_{a+1,a}$}

\begin{leme}\label{ext} Let $L_\lambda,L_\nu\in\mathcal F$,
  $\nu=\lambda-\alpha$ for some isotropic positive root
  $\alpha=\varepsilon_i+\delta_j$ such that 
  $(\lambda+\rho,\alpha)=0$.
Then 
$$[\Gamma_0(G/B,C_{\lambda}):L_{\lambda}]=1,\;\;[\Gamma_0(G/B,C_{\lambda}):L_{\nu}]\geq 1.$$
\end{leme}

\begin{pf} - By Lemma 2 in \cite{VeraCaroI} $\Gamma_0(G/B,C_{\lambda})$ is
  the maximal finite dimensional quotient of the Verma module with
  highest weight $\lambda$. Hence
  $[\Gamma_0(G/B,C_{\lambda}):L_{\lambda}]=1$.
Propositions 2 and 5 in \cite{VeraCaroI} imply that for some parabolic
subgroup $P$ containing $B$ we have
$[\Gamma_0(G/P,L_{\lambda}(P)):L_{\nu}]=1$, where $L_{\lambda}(P)$ is
a simple $P$-module with highest weight $\lambda$. By Lemma 2  in
\cite{VeraCaroI} there is a surjective homomorphism
     $$\Gamma_0(G/B,C_{\lambda})\to  \Gamma_0(G/P,L_{\lambda}(P)).$$
 Hence the statement.
\end{pf}

\begin{leme}\label{lemalpha} - Let $\chi$ and $\theta$ be two distinct central 
characters such that $at(\chi)\geq at(\theta)$. Denote by $T$ the translation 
functor which maps $\mathcal F _{\chi}$ to $\mathcal F _{\theta}$. Then, if 
some dominant weight $\lambda$ (resp. $\lambda _1$, $\lambda _2$) has central
character $\chi$, then $T(L_{\lambda})$ is either a simple module of 
$\mathcal F _{\theta}$ or zero. 

Moreover, if $T(L_{\lambda _1}) = T(L_{\lambda _2})\neq 0$, then 
$\lambda _1 =\lambda _2$.    
\end{leme}

\begin{pf} - All $\goth b$-singular vectors in $L_{\lambda} \otimes E$
have weights of the form 
  $\lambda\pm\varepsilon_j, \lambda, \lambda\pm\delta_k$. At 
most one of those has the central
  character 
$\theta$, as 
one can figure out writing the weight diagrams. Hence the first 
statement. 

Assume now that $T(L_{\lambda _1}) = T(L_{\lambda _2})=L_\mu$, then $\lambda _1
-\lambda _2$ is an isotropic root, a multiple of an even root, or zero. 

Case 1 : $\lambda _1 - \lambda _2$ is a multiple of an even root say $\beta$ (
one has $\lambda _1 - \lambda _2 =\pm 2\varepsilon_i$,
$\varepsilon_i-\varepsilon_j$,$\pm 2\delta_i$ or $\delta_i-\delta_j$)
and then $\lambda _1 +\rho = s_{\beta}(\lambda _2 +\rho)$, 
where  $s_{\beta}$ is the reflection with respect to the root $\beta$. Hence 
$\lambda_1$ and $\lambda_2$ cannot be both dominant except if $(\lambda _1 + \rho , \beta) = 0$,
in which case $\lambda _1 = \lambda _2$.

Case 2 : $\lambda _1 - \lambda _2$ is an isotropic root $\alpha=\pm(\varepsilon_i+\delta_j)$, which
we can assume to be a positive root. As follows from Lemma \ref{ext}
$$[\Gamma_0(G/B,C_{\lambda_1}):L_{\lambda_1}]=1\;\;[\Gamma_0(G/B,C_{\lambda_1}):L_{\lambda_2}]\geq 1.$$
By Corollary 1 in \cite{VeraCaroI} (see also (\ref{tens}))
$$T(\Gamma_0(G/B,C_{\lambda_1}))=\Gamma_0(G/B,C_{\mu}).$$
As $[\Gamma_0(G/B,C_{\mu}):L_\mu]=1$, either $T(L_{\lambda_1})$ or
$T(L_{\lambda_2})$ is zero.
\end{pf}

\begin{leme}\label{lembeta} - Let $\chi$ and $\theta$ be two distinct
central characters such that $at(\chi) \leq at(\theta)$. Then the translation 
functor $T$ which maps $\mathcal F _{\chi}$ to $\mathcal F _{\theta}$ maps
projective indecomposable modules (PIMs for short) to PIMs or to zero.
\end{leme}

\begin{pf} - Let $\lambda$ be a dominant weight with central character 
$\chi$. Since $T(P_\lambda)$ is a projective module in $\mathcal F
  _{\theta}$,
it is sufficient to show that
$$Hom_\goth g(T(P_\lambda), L_{\mu}) = \mathbb C$$
for at most one dominant weight $\mu$, and is zero otherwise. 

Let us denote by $T^*$ the adjoint functor of $T$, which is the translation 
functor mapping $\mathcal F _{\theta}$ to $\mathcal F _{\chi}$; one has 
$$Hom_{\mathfrak g}(T(P_\lambda), L_{\mu}) = Hom_{\mathfrak g}(P_\lambda,T^*( L_{\mu}))$$
and by the Lemma \ref{lemalpha}, the statement follows.
\end{pf}

\subsection{The action of the switch functor on simple modules and PIMs}

\begin{leme}\label{switch} - Let $\lambda$ be an integral dominant
  weight.

If $f_\lambda$ has the empty tail position, then
$sw(L_\lambda)=L_\lambda$. 

If $f_\lambda$ has $>$ or $<$ at the tail position, then
$sw(L_\lambda)=0$. 

If $f_\lambda$ has several $\times$ at the tail position and no core symbols, then
$sw(L_\lambda)=L_\mu$, where $f_\mu$ is obtained
from $f_\lambda$ by  change of sign. 

The action of the switch functor on PIMs is given by the same rule.

\end{leme}

\begin{pf} - The first assertion easily follows from
$$sw(\Gamma_0(G/B,C_\lambda))=\Gamma_0(G/B,C_\lambda).$$

To prove the second assertion choose the parabolic subalgebra $\goth q\subset\goth g$
with semisimple part $\goth{osp}(2k+3,2k)$ if $f_\lambda$ has
$k$ $\times$ and $>$ at the tail position and with  semisimple part $\goth{osp}(2k+1,2k+2)$ if $f_\lambda$ has
$k$ $\times$ and $<$ at the tail position. Then $C_\lambda$ has a
natural $\goth q$-module structure and $L_\lambda$ is a quotient of the
parabolically induced module $S_\lambda=U(\goth g)\otimes_{U(\goth  q)}C_\lambda$. It is 
easy to check that 
$$sw(S_\lambda)=(U(\goth g)\otimes_{U(\goth  q)}(C_\lambda\otimes E))_{\gamma(\lambda)}=0.$$
Hence $sw(L_\lambda)=0$. 

Now let $f_\lambda$ have $k$ $\times$-s at the tail position and no core
symbols. Let $\goth q$ be the parabolic subalgebra with semisimple
part $\goth{osp}(2k+1,2k)$. If the sign of $f_\lambda$ is $-$ then 
$L_\lambda$ is a quotient of $S_\lambda$. On the other hand,
$$S'_\lambda:=sw(S_\lambda)=U(\goth g)\otimes_{U(\goth  q)}(C_\lambda\otimes E'),$$
where $E'$ is the standard module over semisimple part of $\goth q$. The unique simple quotient of $S'_\lambda$
is isomorphic to $L_\mu$, where $f_\mu$ is obtained from $f_\lambda$
by change of sign. Since the application of the switch functor to
any other simple subquotient of $S_\lambda$ can not produce $L_\mu$ we
have $sw(L_\lambda)=L_\mu$. If the sign of
$f_\lambda $ is $+$, the assertion follows similarly from the fact
that $sw(S'_\lambda)=S_\lambda$.

Finally, the statement about PIMs follows by duality. 
\end{pf}

\subsection{Elementary changes}\label{ec}
We define elementary changes on the weight diagram by the list.

$\bullet$ The change of the sign $(+)$ and $(-)$ in  front of the diagram.

$\bullet$ An elementary change which involves positions $a$ and $a+1$
with $a>\frac{1}{2}$:

a) $\ldots \times \; \circ \ldots \rightarrow \ldots > \; < \ldots$ (decreases the degree of atypicality by 1),

b) $\ldots < \; \circ \ldots \leftrightarrow  \ldots \circ  \; < \ldots$ (doesn't change the degree of atypicality),

c) $\ldots > \; \circ \ldots \leftrightarrow \ldots \circ  \; > \ldots$ (doesn't change the degree of atypicality),

$\bullet$ An elementary change which involves positions $\frac{1}{2}$ and $\frac{3}{2}$

at) $(+)\times^k \; \circ \ldots \rightarrow {}^>_{\times^{k-1}}\; < \ldots$ (decreases the degree of atypicality by 1),

bt) ${}^>_{\times^{k}}\; \circ \ldots \leftrightarrow (-)\times^k  \; > \ldots$ (doesn't change the degree of atypicality),

ct)  ${}^<_{\times^{k}}\; \circ \ldots \leftrightarrow (-)\times^k  \; < \ldots$ (doesn't change the degree of atypicality),

\noi The $\times^k$ sign indicates that there are $k$ $\times$ at the
tail position and $\leftrightarrow$ means that we can go in either
direction.

\begin{leme}\label{lemdelta} - If $(f_{\lambda}, f_{\mu})$ is a pair
  of weight diagrams of the list a),..., ct), then for a suitable
  choice of a translation functor $T$ we have $T(L_{\lambda})= L_{\mu}$ and $T^*(P_{\mu}) = P_{\lambda}$.
\end{leme}

\begin{pf} - In the cases b),c),bt) and ct), we already know
  (\cite{VeraCaroI} Section 6) that $T$ is an equivalence of the
  corresponding blocks, thus we have nothing to prove.

It remains to prove the statement for a pair 
$(f_{\lambda}, f_{\mu})$ in the cases a)and at). 

We have $\mu = \lambda + \delta _{a+1}$. It is easy to see that among
the weights $\lambda\pm\varepsilon_i,\lambda\pm\delta_j$ only $\mu$ has
weight $\gamma(\mu)$. Therefore by Corollary 1 in \cite{VeraCaroI} we
have
\begin{equation}\label{equaux}
T(\Gamma_0(G/B,C_\lambda))=\Gamma_0(G/B,C_\mu).
\end{equation}
The multiplicity of $L_\lambda$ in  $\Gamma_0(G/B,C_\lambda)$ is 1 as
well as the multiplicity of  $L_\mu$ in  $\Gamma_0(G/B,C_\mu)$. If
$L_\nu$ is a simple subquotient of  $\Gamma_0(G/B,C_\lambda)$ and
$\nu\neq\mu$, then $\nu<\mu$ and by Lemma \ref{lemalpha} $T(L_\nu)\neq L_\mu$.  
Therefore (\ref{equaux}) implies  $T(L_\lambda)= L_\mu$.
\end{pf}

\begin{leme}\label{lemgamma} - Let $\lambda$ be any dominant weight with atypicality degree $k>0$ . Then one can find a dominant typical weight $\mu$ such that
there exist a sequence of weights $\mu = \mu _1, \mu _2, \ldots , \mu_r = \lambda$ and translation functors $T_1, \ldots, T_{r-1}$ such that
$P_{\mu _i} = T_{i-1}(P_{\mu _{i-1}})$, and $at(\mu _i) \leq at(\mu_{i+1})$. 
\end{leme}
\begin{pf} - Due to the previous Lemma we have to check that $f_\lambda$
can be transformed to a typical $f_\mu$ by elementary changes.

We prove the statement by induction on degree of atypicality of
$\lambda$. Let $t$ be the position of the rightmost $\times$  in the weight
diagram $f_{\lambda}$. Assume first that $t\neq \frac{1}{2}$. If the
position $t+1$ is empty we 
can use elementary change a) to decrease the atypicality degree
of $\lambda$. If the position $t+1$ is occupied by a core symbol, we
can use elementary changes  of type b) and c) to move all core symbols in positions
$t+1,t+2,\ldots$ to the right. The diagram obtained in this way will have
the position $t+1$ empty and now we can decrease the degree of
atypicality using elementary change a).

Now let $t=\frac{1}{2}$. That means $f_\lambda$ has only core
symbols outside the tail position. Using elementary changes of type b)
and c) we can transform $f_\lambda$ to the diagram that has an empty
position $\frac{3}{2}$. Hence without loss of generality we may assume
that the position $\frac{3}{2}$ in $f_\lambda$ is empty. Now we are
going consider 3 different cases.

If all symbols at the tail positions are $\times$ and the sign is $+$,
we apply elementary change at) to decrease the degree of atypicality.

If all symbols at the tail positions are $\times$ and the sign is $-$,
we apply the switch functor to $f_\lambda$ and reduce the situation to
the previous case.

If the tail position contains a core symbol, we apply  elementary change
bt) or ct) and reduce the situation to the previous case.

Thus, the statement follows by induction.
\end{pf}

\section{The action of translation functors in the  case $SOSP(2m,2n)$}

In this section $G=SOSP(2m,2n)$.
There is an involutive automorphism $\sigma$ preserving the maximal
torus $H$ which acts on $\Lambda$ by the formula
$$\sigma(\varepsilon_m)=-\varepsilon_m,\;\;\sigma(\varepsilon_i)=-\varepsilon_i\;\;\text{if}\;\;i\neq m\;\;
\sigma(\delta_j)=\delta_j.$$
The induced action of $\sigma$ on the weight diagrams is the change of
the sign if the diagram has a sign, otherwise $\sigma$ preserves the weight diagram.

The action of $\sigma$ on $\mathcal F$ preserves blocks except the
case when $\chi$ is typical and the diagram has an empty
tail position. In the latter case $\sigma$ permutes two
blocks. If $M\in\mathcal F$ we denote by $M^\sigma$ the twist of $M$
by $\sigma$.

The induced action of $\sigma$ on the Grothendieck ring $\mathcal K(\mathcal F)$
has two eigenspaces  $\mathcal K(\mathcal F)^{\pm}$ with eigenvalues
$\pm 1$. So we have the following decompositions 
$$\mathcal K(\mathcal F)=\mathcal K(\mathcal F)^+\oplus\mathcal K(\mathcal F)^-$$
and
$$\mathcal K(\mathcal E)=\mathcal K(\mathcal E)^+\oplus\mathcal K(\mathcal E)^-,\,\mathcal K(\mathcal P)=\mathcal K(\mathcal P)^+\oplus\mathcal K(\mathcal P)^-.$$
Note that the translation functors preserve those splittings.

\subsection {Categorification of $\goth{gl}_{\infty}$} 
As in Section \ref{cat} we realize translation functors by certain
linear operators in a representation of
$\mathfrak{gl}_{\infty}$.  By $\mathfrak{gl}_{\infty}$ we understand
the Lie algebra with Dynkin diagram
$$\ldots-\circ - \circ - \circ - \circ  \ldots ,$$ and by $U$ its standard
representation with basis 
$u_{i}$ with $i\in\mathbb Z$. Let $E_{i,j}$ be the element of 
$\mathfrak{gl}_{\infty }$ which acts on $V$ the following way:
$E_{i,j}(v_k) = \delta _{jk}v_i$. Denote the weight of $v_a$ by
$\gamma_a$ for $a>0$ and by $-\gamma_{-a}$ for $a<0$. Set weight of $v_0$
to be zero.
Let $U^{\leq 0}$ be the span of $u_i$ for all $i\leq 0$, $U^+$ and $U^-$
be the span of $u_i$ for $i>0$ and $i<0$ respectively. Set
$$X^+=\Lambda^m(U^{\leq 0})\otimes\Lambda^n(U^+),$$
$$X^-=\Lambda^m(U^{-})\otimes\Lambda^n(U^+).$$

Finally set 
$$F_{i,j}=(-1)^{i+j+1}(E_{i,j}+E_{-j,-i}),\,\,\text {for}\,\, i,j>0;\,\,F_{i,0}=2E_{0,-i}, F_{0,i}=E_{-i,0}.$$
Those elements generate the Lie algebra $\mathfrak{sl}_{{\infty}/2}\oplus\mathfrak{sl}_{\infty/2}$ 
inside $\mathfrak{gl}_{\infty}$.

Let $\lambda+\rho\in\Lambda^+$ be
such that 
$\lambda + \rho = \sum _{i=1} ^m a_i \varepsilon _i + \sum _{j=1}  ^n b_j \delta _j$ 
with $a_i\in\mathbb Z_{\geq 0},b_j \in \mathbb Z_{>0}$, we associate to $\lambda$ the
vector 
$$x_\lambda:= u_{-a_m}\wedge \ldots \wedge u_{-a_{1}} \otimes u_{b_1}\wedge \ldots \wedge u_{b_n}.$$

The weight of $x_\lambda$ and  the translation functors $T_{a,a+1}$
and  $T_{a+1,a}$ (for $a\in\mathbb Z_{>0}$) are defined as in  Section \ref{cat}.
In addition we define the translation functors 
$T_{0,1}, T_{1,0}:\mathcal F\to \mathcal F$ by the formulae: 
$$T_{0,1}(M) = (M\otimes E)_{\gamma- \gamma _{1}}\;\;\text{for}\;\;M\in\mathcal F_\gamma$$
$$T_{1,0}(M) = (M\otimes E)_{\gamma+\gamma _{1}}\;\;\text{for}\;\;M\in\mathcal F_\gamma$$
Next we define isomorphisms of $\mathbb Z$-modules $\psi^\pm:\mathcal K(\mathcal E)^\pm\to X^\pm$
by setting 
$$\psi^+(\mathcal E(\lambda))=x_\lambda,\,\,\psi^-(\mathcal E(\lambda))=0$$
if $a_m=0$,
$$\psi^\pm(\mathcal E(\lambda)\pm\mathcal E(\lambda)^\sigma)=x_\lambda,$$
if $a_m>0$.

\begin{leme} \label{Lemmaa1} - For any $a>0$ one has 
$$\psi^\pm \circ T_{a+1,a} = F_{a+1,a}\circ \psi^\pm$$
and
$$\psi^\pm \circ T_{a,a+1} = F_{a,a+1}\circ \psi^\pm .$$ 
If $a=0$, then 
$$\psi^+ \circ T_{1,0} = F_{1,0}\circ \psi^+$$
and
$$\psi^+ \circ T_{0,1} = F_{0,1}\circ \psi^+ .$$
\end{leme}
\begin{proof}
This lemma can be proven exactly as
Lemma \ref{Lemmaa} by direct comparison. 

The action of the functors $T_{a,a+1}$ for $a\geq 1$ in terms of
weight diagrams is given in figure 1.

Below in figure 2 and figure 3 we give the action of $T_{0,1}$ and
$T_{1,0}$ respectively.

\end{proof}
\bigskip

$$\begin{array}{ccccccccccc}
f_{\lambda} && 0 & 1  & f_{\mu}& &  0 & 1 &\\
&&&&&\\
& [\pm]&\circ & \circ &&& \emptyset &&&&\\
&[\pm]&\circ & < &&& \emptyset&&&&\\
& [\pm]&\circ &>&&&\emptyset&&&&\\
&[\pm]& \circ & \times &&& \emptyset&&&&\\
&& > & < &&[+]&\circ&\times&\oplus &[-]\circ& \times\\
& &> & \circ &&[+]&\circ&>&\oplus &[-]\circ&>\\
& &> & > &&& \emptyset&&&&\\
& &> & \times &&& \emptyset&&&&\\
&&&&(figure \; 2)&&&&
\end{array}$$

\bigskip

$$\begin{array}{cccccccccc}
f_{\lambda} && 0 & 1 & f_{\mu} & 0 & 1 &\\
&&&&&\\
& [\pm]&\circ & \circ && \emptyset &&&&\\
&[\pm]&\circ & < && \emptyset&&&&\\
& [\pm]&\circ & >&&>&\circ &&&\\
&[\pm] &\circ & \times && >&<&&&\\
&& > & < &&\emptyset&&&&\\
& &> & \circ &&\emptyset&&&&\\
& &> & > && \emptyset&&&&\\
& &> & \times && \emptyset&&&&\\
&&&(figure \; 3)&&&&
\end{array}$$

\bigskip

\subsection{ Action of translation functors on simple modules and PIMs}
 The following statement is analogous to Lemma \ref{lemalpha}.
\begin{leme}\label{lemalphaev} - (a) Let $\chi$ and $\theta$ be two distinct central 
characters such that $at(\chi)\geq at(\theta)$. Assume that $a\geq 1$. Let $T=T_{a,a+1}$ or
$T_{a+1,a}$ be a translation functor which maps  
$\mathcal F _{\chi}$ to $\mathcal F _{\theta}$. Then, if 
some dominant weight $\lambda$ (resp. $\lambda _1$, $\lambda _2$) has central
character $\chi$, then $T(L_{\lambda})$ is either a simple module of 
$\mathcal F _{\theta}$ or zero. 

Moreover, if $T(L_{\lambda _1}) = T(L_{\lambda _2})\neq 0$, then 
$\lambda _1 =\lambda _2$.    

(b) If  $at(\chi)\leq at(\theta)$ and $a\geq 1$, then $T(P_\lambda)$
is either PIM or zero and if $T(P_{\lambda _1}) = T(P_{\lambda _2})\neq 0$, then 
$\lambda _1 =\lambda _2$. 
\end{leme}

We define non-tail elementary changes as a), b), c) in subsection
\ref{ec} with the convention that a non-tail elementary change does not
change the sign of the diagram.

\begin{leme}\label{simplev} - Let $\lambda$ be a dominant
  weight. Assume that the degree of atypicality of $\lambda$
  is not less than  the degree of atypicality of
  $\gamma(\lambda)-\gamma_1$.
If
  $f_\lambda$ does not have $>$ at the tail position, then
  $T_{0,1}(L_\lambda)=0$.
If $f_\lambda$ has $>$ at the tail position, then the action of
$T_{0,1}$ 
is given by the following change in the diagram $f_\lambda$.

1) $ {}^>_{\times^k}\circ\ldots\to\times^{k}>\ldots;$

2) ${}^>_{\times^{k-1}}\times\ldots\to\times^{k}>\ldots;$

3)  $>\circ\ldots\to\;[+]\circ>\ldots\oplus [-]\;\;\circ >\ldots.$

Assume that the degree of atypicality of $\lambda$
  is not less than  the degree of atypicality of
  $\gamma(\lambda)+\gamma_1$.
If $T_{1,0}(L_\lambda)\neq 0$, then $f_\lambda$ has  no $>$ at the tail
position and $f_\lambda(1)=\circ$ or $>$. The action of $T_{1,0}$ is given by the
following change in $f_\lambda$.

4)  $\times^k\circ\ldots\to {}^>_{\times^{k-1}}<\ldots;$ 

5)  $\times^k >\ldots\to {}^>_{\times^{k}}\circ\ldots;$ 

6)  $[\pm]\circ >\ldots\to\;>\circ\ldots.$
\end{leme}
\begin{pf} The statement is an immediate consequence of Theorem 3.2
  (iii) in \cite{SLMN}.
\end{pf}

\begin{coro}\label{pimev} -  Let $\lambda$ be dominant and with degree
  of atypicality not greater than  the degree of atypicality of
  $\gamma(\lambda)+\gamma_1$. Then $T_{1,0}(P_\lambda)=0$ if $f_\lambda$
  contains $>$ at the tail position. If  $f_\lambda$ does not
  have $>$ at the tail position the action   of $T_{1,0}$ can be
  described by the following diagrams.

a') $\times^{k}>\ldots\to  {}^>_{\times^k}\ldots\oplus  {}^>_{\times^{k-1}}\times\ldots;$

b') $[\pm]\circ >\ldots\to >\circ\ldots.$

Now assume that  the degree
  of atypicality of  $\lambda$ is not greater than  the degree of atypicality of
  $\gamma(\lambda)-\gamma_1$. If $T_{0,1}(P_\lambda)\neq 0$ then
  $f_\lambda$ contains $>$ at the tail position.  The
  action of $T_{0,1}$ on $P_\lambda$ is given by one of the following changes:

c') $ {}^>_{\times^{k-1}}<\ldots\to \times^k\circ\ldots;$

d') ${}^>_{\times^{k}}\circ\ldots \to \times^k >\ldots;$

e')  $>\circ\ldots\to\;[+]\circ>\ldots\oplus [-]\;\;\circ >\ldots.$

\end{coro}

We call a')-e') the elementary tail changes.
Note that a translation functor corresponding to an elementary tail
change does not always map a PIM to a PIM, a') and e') map sometimes 
a PIM to the direct sum of two PIMs. Thus, a straightforward analogue of Lemma \ref{lemgamma} 
does not hold. 
\begin{leme}\label{uniqueev} - The Grothendieck group $\mathcal K(P)$ is generated by
  $[P_\lambda]$ for typical $\lambda$ and 
  $[T_m\circ\dots\circ T_1(P_\lambda)]$ ($\lambda$ typical) where $T_i$
    does not decrease the degree of atypicality of  $T_{i-1}\circ\dots\circ T_1(P_\lambda)$. 
\end{leme}

\begin{pf} - We prove the statement by induction on the degree of
  atypicality. Let $S$ denote the span of $[T_m\circ\dots\circ T_1(P_\lambda)]$.
It is clear that $[P_\lambda]\in S$ for typical $\lambda$. Assume
that  $[P_\mu]\in S$ if  $at(\mu)=k-1$. Suppose that
$at(\lambda)=k$. If $f_\lambda$ has at least one $\times$ which is not
at the tail position, then we can obtain $[P_\lambda]$ from some $[P_\nu]$
with $at(\nu)=k-1$ in the same way as in the proof of Lemma \ref{lemgamma}.

Now assume that  all the $\times$ of $f_\lambda$ are at the tail
position. We check all the possible cases for $f_\lambda$.

If $f_\lambda=\times^k\circ\ldots$, then $P_\lambda=T_{0,1}(P_\nu)$
with $f_\nu= {}^>_{\times^{k-1}}<\ldots$ (Corollary \ref{pimev} c')). 
By induction assumption $[P_\nu]\in S$. Hence $[P_\lambda]\in S$.

If $f_\lambda=\times^k >\ldots$ or $f_\lambda=\times^k <\ldots$, we
use non-tail elementary changes to move the non-tail symbols to the right,
then apply the translation functor as in the previous case and then
move the non-tail symbols back. For instance, if $f_\lambda=\times^k >\circ$, we use 
$${}^>_{\times^{k-1}}<>\to \times^k\circ>\to\times^k >\circ.$$

Now let $f_\lambda= {}^>_{\times^{k}}\ldots$. Using non-tail
elementary changes we can reduce to the case  $f_\lambda={}^>_{\times^{k}}\circ\ldots$ in the same way as above.
By Corollary \ref{pimev} a')
$$T_{1,0}(P_\nu)=P_\lambda\oplus P_\mu,$$
where $f_\nu=\times^k >\ldots$ and $f_\mu={}^>_{\times^{k-1}}\times\ldots$. But 
$f_\mu$ has $\times$ at non-tail position. We have proved above
that $[P_\mu]\in S$. We also have checked above that $[P_\nu]\in S$. This implies
$[P_\lambda]\in S$.
\end{pf}

\section{PIM as a linear combination of Euler characterstics}

\subsection{The case of a tailless weight} In this subsection
$G=SOSP(2m,2n)$ or $SOSP(2m+1,2n)$.

\begin{theo}\label{lem1} - Let $\lambda$ be a tailless dominant weight. One has:
$$[P(\lambda)] = \sum _{\mu \in \mathcal P (\lambda)}  \mathcal E (\mu).$$
\end{theo}
\begin{pf} - This lemma has the same proof as the corresponding
  statement, due to Jonathan Brundan, in the case
  $\mathfrak{gl}(m,n)$. However,
we write down the argument in our setting. 

Due to Lemma \ref{lemdelta} and Lemma \ref{uniqueev} it is sufficient to check that if the statement holds
for $P_\kappa$, then it holds for $P_\lambda=T(P_\kappa)$, where $T$ is a
translation functor corresponding to some elementary change. If the
elementary change is of type b) or c) consisting of moving a core
symbol from position $t+1$ to $t$, then clearly the weight diagrams of
$\mathcal P(\lambda)$ are obtained from those of $\mathcal P(\kappa)$
by exchanging symbols in position in $t+1$ and $t$. If the elementary
change is of type a) 
$$>\,< \to \times\,\circ,$$
then the cap diagram of $f_\lambda$ has exactly one new cap joining $\times$
and $\circ$ in $f_\lambda$. All other caps remain the same. Hence the
statement holds in this case as well.
\end{pf}
\begin{exs} Let $G=SOSP(7,6)$.

$$\begin{array}{ll}\left[P(\circ \; \times \; > \; < \; \times)\right]
    & = \mathcal E (\circ \; \times \; > \; < \; \times) + \mathcal E
    (\circ \; \times\; > \; < \; \circ \; \times) + \\
 & +\mathcal E(\circ \; \circ \; > \; < \; \times \; \circ \; \times ) + \mathcal E (\circ\; \circ\; >\; <\; \circ\; \times \; \times),\end{array}$$
 
$$\begin{array}{ll}\left[P(\circ \; \times \; \times \; \circ \; \times)\right]  & = \mathcal E (\circ \; \times \; \times \; \circ \; \times) + \mathcal E (\circ \; \times \; \circ\; \times \; \times) + \mathcal E(\circ \; \times \; \times \; \circ \; \circ \; \times) + \\
 & +\mathcal E (\circ\; \times\; \circ\; \times\; \circ\; \times)  + \mathcal E (\circ\; \circ\; \times \; \circ\; \times \; \circ\;\times) + \mathcal E (\circ\; \circ\; \circ\;  \times\; \times \; \circ\;\times) +\\
 &+ \mathcal E (\circ\; \circ\; \times \; \circ \; \circ\;\times \; \times) + \mathcal E (\circ\; \circ\; \circ\; \times \; \circ\; \times \;\times).\end{array}$$
 and the caps are the following: 
 
 
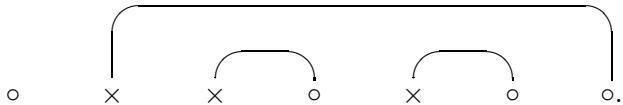
\xymatrix{
 & & & & & & \\
{\circ} & {\times}\ar@{-} `u[ur] `[rrrrr]  [rrrrr] & \times\ar@{-} `u[r] `[r]  [r] & \circ & \times\ar@{-} `u[r] `[r]  [r] & \circ & \circ .}

~\\

 \end{exs}

\subsection{The general case $G=SOSP(2m+1,2n)$.} Now let $\lambda$ have a tail. Color all
$\times$ at the tail position emerald. We define
the tailless weight $\bar{\lambda}$ as follows. Ignore for a moment the
tail position and consider the cap diagram associated to the
weight diagram for the remaining positions. We call a position free if it
is empty and is not an end of a cap. Now, we move $\times$-s from the tail
position to free positions according to the rule below:

$\bullet$ If $f_\lambda$ has a core symbol at the tail, move all
$\times$ from the tail position to the free positions number 
1,3,... counting from the left.

$\bullet$ If $f_\lambda$ does not have a core symbol at the tail, move all
but one $\times$ from the tail to the free positions number
2,4,... counting from the left.

\begin{ex} If $f_\lambda=(\times^3\times\circ\ldots)$, then
 $f_{\bar\lambda}=(\times\times\circ\circ\times\circ\times\circ\ldots)$.
\end{ex}

Let 
$$[P(\mu)] = \sum _{\nu} a(\mu,\nu)\mathcal E (\nu).$$ 
As follows from Theorem \ref{lem1}, if $\mu$ is tailless $a(\mu, \nu) = 0$ or $1$.

\begin{theo}\label{prop1} -  Let $G=SOSP(2m+1,2n)$. If the tail position of the diagram
  $f_{\lambda}$ contains $<$,  $>$ or  $(-)$ sign, then 
$$a(\lambda,\nu)=(-1)^{c(\lambda, \nu)}a(\bar{\lambda}, \nu)$$
where $c(\lambda,\nu)= x +y$ where $x$ is the total number of emerald
$\times$ in $f_{\bar{\lambda}}$,
and $y$ is the number of emerald $\times$ in $f_{\bar{\lambda}}$ 
moved along the caps in order to get $f_{\nu}$ from $f_{\bar{\lambda}}$.

If the tail position of $f_{\lambda}$ has a $(+)$, we change the sign
of $a(\lambda,\nu)$ for all $\nu$ such that $f_\nu$ has 
a $\times$ at the tail position.
\end{theo}
\begin{pf} - As in the proof of Theorem \ref{lem1} we have to check that
  the statement for $P_\mu$ implies the statement for $T(P_\mu)$ for a
  translation functor $T$ corresponding to some elementary
  change. This check for elementary changes  a)-c) is completely
  analogous to the case of tailless $\lambda$. So we leave it to the
  reader.

Let $\lambda$ and $\mu$ be related by an elementary change b) i.e.
$$f_\mu=(-)\times^k\;>\ldots\rightarrow 
f_\lambda={}_{\times^k}^>\;\circ\ldots$$
Then $\bar{\lambda}$ is obtained from  $\bar{\mu}$ by switching
$\times$ at the tail position with $>$ at position $\frac{3}{2}$, and
the number of emerald $\times$ in $f_{\bar{\lambda}}$ and in  
$f_{\bar{\mu}}$ is the same.
All $\nu'\in \mathcal P(\bar\lambda)$ are obtained from  
$\nu\in \mathcal P(\bar\mu)$ by interchanging symbols at positions
$\frac{1}{2}$ and  $\frac{3}{2}$. Clearly $c(\lambda,\nu)=c(\mu,\nu)$.
The case of elementary change ct) is similar.

Now let $\lambda$ and $\mu$ be related by an elementary change a), namely
$$f_\mu={}_{\times^k}^>\;<\ldots\rightarrow 
f_\lambda=(+)\times^{k+1}\;\circ\ldots$$
Then $\bar{\lambda}$ is obtained from  $\bar{\mu}$ by
removing core symbols from $\frac{1}{2}$ and  $\frac{3}{2}$ and adding
$\times$ to the tail position  $\frac{1}{2}$. The cap diagram for  
$f_{\bar{\lambda}}$ has an additional cap joining  $\frac{1}{2}$ and
$\frac{3}{2}$. The number of emerald
$\times$ increases by $1$. However, since sign of $f_\lambda$ is $(+)$,
the signs agree after applying the switch functor to $P_\lambda$.
\end{pf}

\begin{exs} $G=SOSP(5,4)$

$\left[P \left(\begin{array}{cc} > & \\ \times & < \\ \end{array}\right)\right]= -\mathcal E (> \; < \; \times) + \mathcal E (> \; < \; \circ \; \times),$

$\left[P \left((+)\begin{array}{c} \times \\ \times \\ \end{array}\right)\right]= -\mathcal E ((+) \times \; \circ \; \times) - \mathcal E (\circ \; \times \; \times) + \mathcal E ((+) \times \; \circ \; \circ \; \times) + \mathcal E (\circ \; \times \; \circ \; \times), $

$\left[P \left((-)\begin{array}{c} \times \\ \times \\ \end{array}\right)\right]= \mathcal E ((+)\times \; \circ \; \times) - \mathcal E (\circ \; \times \; \times) - \mathcal E ((+)\times \; \circ \; \circ \; \times) + \mathcal E (\circ \; \times \; \circ \; \times) .$

\end{exs}

\subsection{The general case $G=SOSP(2m,2n)$.} 
Let $\lambda$ be a dominant weight for  $SOSP(2m,2n)$. Define a dominant weight
$\lambda'$ for  $SOSP(2m+1,2n)$ whose weight diagram satisfies the condition
$f_{\lambda'}(t) = f_{\lambda}(t- \frac{1}{2})$.
If $f_\lambda$ has some $\times$ at the tail and no $>$, then
$f_{\lambda'}$ must have a sign. In this case we use the notation $\lambda'_\pm$
depending on the sign. 

\begin{theo}\label{prop1ev} -  Let $G=SOSP(2m,2n)$. Then for any
  $\nu\in\Lambda^+-\rho$ we have
$$a(\lambda,\nu)=a(\lambda'_{\pm},\nu').$$
\end{theo} 
\begin{rem} As follows from Theorem \ref{prop1} $a(\lambda'_{+},\nu')=a(\lambda'_{-},\nu')$
since $f_{\nu'}$ can not have $\times$ at the tail position.
Thus, Theorem \ref{prop1ev} implies that the coefficient
$a(\lambda,\mu)$ for $SOSP(2m,2n)$ can be computed using the algorithm
for $SOSP(2m+1,2n)$. If $f_\lambda$ has $>$ at the tail the coefficients remain
the same. Otherwise some coefficients disappear because $f_{\nu'}$
can not have a $\times$ at the tail position.  
\end{rem}
\begin{pf} 
Let us introduce a few notations. For any $\mathbb Z$-module $A$,
denote $_\mathbb QA=\mathbb Q\otimes_\mathbb Z A$. For $m,n$ fixed ,
we denote by $\mathcal K^+_{ev}(\mathcal F)$ the
$\sigma$-invariant subgroup of the Grothendieck group for $SOSP(2m,2n)$
and by $\mathcal K_{odd}(\mathcal F)$ the Grothendieck group for $SOSP(2m,2n)$.
The definitions of  $\mathcal K^+_{ev}(P)$,  $\mathcal K^+_{ev}(\mathcal E)$,
$\mathcal K_{odd}(P)$ and $\mathcal K_{odd}(\mathcal E)$ are obvious.

We define two $\mathbb Q$-linear maps 
$\alpha:{} _\mathbb Q\mathcal K_{odd}(\mathcal E)\to{}_\mathbb Q\mathcal K^+_{ev}(\mathcal E)$
and
$\beta:{}_\mathbb Q\mathcal K^+_{ev}(\mathcal E)\to{}_\mathbb Q\mathcal K_{odd}(\mathcal E)$
in the following way.
First we define $\alpha:U\to V\oplus V^*$ and $\beta: V\oplus V^*\to U$ by the formulae  
$$\alpha(v_{1/2})=0,\,\alpha(v_i)=u_{i-1/2},\,i>1,\,\,\alpha(w_i)=u_{1/2-i}$$
$$\beta(u_i)=w_{1/2-i}, i\leq 0,\,\, \beta(u_i)=v_{i-1/2}, i>0.$$
Next, we extend $\alpha$ and $\beta$ to
$\Lambda^m(V^*)\otimes\Lambda^n(V)$ and $X^+$ in the natural way. The
following diagram explains the maps $\alpha$ and $\beta$ on the level
of Grothendieck groups
$$\begin{array}{rcl}
&\xrightarrow{\beta}&\\
X^+&\xleftarrow{\alpha}&\Lambda^m(V^*)\otimes\Lambda^n(V)\\
\psi^+\uparrow&&\uparrow\phi\\
&\xrightarrow{\beta}&\\
_\mathbb Q\mathcal K^+_{ev}(\mathcal E)&\xleftarrow{\alpha}&_\mathbb Q\mathcal K_{odd}(\mathcal E).\\
\end{array}$$
Note that $\alpha$ is surjective, $\beta$ is injective and $\alpha\circ\beta=\operatorname{id}$.
For a linear operator $T$ in $\mathcal K^+_{ev}(\mathcal E)$
corresponding to a translation functor we define an operator $T'$ in  $\mathcal K_{odd}(\mathcal E)$
by the following rules

if $T=T_{a,a+1}$ then $T'=T_{a+1/2,a+3/2}$,

if $T=T_{a,a-1}$ and $a\neq 1$ then $T'=T_{a+1/2,a-1/2}$,

if $T=T_{1,0}$ then $T'=2T_{3/2,1/2}$.

Direct computation gives the following result:
\begin{leme}\label{rel1} One has for any $T=T_{a,a\pm 1}$
$$T=\alpha\circ T'\circ\beta.$$
\end{leme}

Next we define $\bar{\beta}:{}_\mathbb Q\mathcal K^+_{ev}(P)\to{}_\mathbb Q\mathcal K_{odd}(P)$
by setting:

$\bar\beta[P_\lambda]=\frac{1}{2}([P_{\lambda'_+}]\oplus [P_{\lambda'_-}])$ 
if $f_\lambda=(\times^k\ldots)$, $k>0$;

$\bar\beta[P_\lambda\oplus P^\sigma_\lambda]=[P_{\lambda'}]$ if 
$f_\lambda=(\circ\ldots)$;

$\bar\beta[P_\lambda]=[P_{\lambda'}]$ if $f_\lambda=({}_{\times^k}^>\ldots)$;

and  $\bar{\alpha}:{}_\mathbb Q\mathcal K_{odd}(P)\to{}_\mathbb Q\mathcal K^+_{ev}(P)$
as follows:

If $\nu=\lambda'_{\pm}$ for some dominant $SOSP(2m,2n)$ weight
$\lambda$ we set $\bar\alpha[P_\nu]=[P_\lambda]$ or $[P_\lambda\oplus  P^\sigma_\lambda]$ 
(in the case $P_\lambda\neq  P^\sigma_\lambda$). 
Otherwise we set $\bar\alpha[P_\nu]=0$.

It follows immediately from the definitions that $\bar\alpha$ is
surjective, $\bar\beta$ is injective and $\bar\alpha\circ\bar\beta=\operatorname{id}$.

\begin{leme}\label{rel2} Let $T=T_{a,a\pm 1}$, $P=P_\lambda$ or
  $P=P_\lambda\oplus P^\sigma_\lambda$. If $T$ does not decrease the
  degree of atypicality of $\lambda$, then
$$T([P])=\bar\alpha\circ T'\circ\bar\beta([P]).$$
\end{leme}
\begin{pf} If $T\neq T_{0,1}$ or $T_{1,0}$ the statement follows
  from the fact that the elementary changes listed in subsection
  \ref{ec} are the same in the even and odd cases.

 If $T\neq T_{0,1}$ or $T_{1,0}$ one can just compare the actions of
 $T$ and $T'$ in all possible cases. We show how it works in the most
 interesting cases and leave to the reader the remaining cases.

$$\begin{array}{lcl}
\times^k<&\xrightarrow{\bar\beta}&\frac{1}{2}((+)\times^k<\oplus(-)\times^k<)\\
T_{0,1}\downarrow &&\downarrow T_{1/2,3/2}\\
0&\xleftarrow{\bar\alpha}&\frac{1}{2}({}_{\times^k}^<)\circ
\end{array}$$

\vskip 0.5cm

$$\begin{array}{lcl}
{}_{\times^{k-1}}^><&\xrightarrow{\bar\beta}&{}_{\times^{k-1}}^><\\
 T_{0,1}\downarrow&&\downarrow T_{1/2,3/2}\\
\times^k\circ&\xleftarrow{\bar\alpha}&(+)\times^k\circ
\end{array}$$

\vskip 0.5cm

$$\begin{array}{lcl}
\times^k>&\xrightarrow{\bar\beta}&\frac{1}{2}((+)x^k>\oplus(-)x^k>)\\
T_{1,0}\downarrow &&\downarrow 2T_{3/2,3/2}\\
{}_{\times^k}^>\circ\oplus{}_{\times^{k-1}}^>\times&\xleftarrow{\bar\alpha}&{}_{\times^k}^>\circ\oplus{}_{\times^{k-1}}^>\times
\end{array}$$

\vskip 0.5cm

$$\begin{array}{lcl}

[+]\circ>\oplus[-]\circ>&\xrightarrow{\bar\beta}&\circ>\\
\downarrow T_{1,0}&&\downarrow 2T_{3/2,1/2}\\
2>\circ&\xleftarrow{\bar\alpha}&2>\circ
\end{array}$$

\end{pf}
 
\begin{leme}\label{rel3} One has 
$$\beta\circ\alpha|_{{}_\mathbb Q\mathcal K_{odd}(P)}=\bar\beta\circ\bar\alpha.$$
\end{leme}
\begin{pf} Both $\beta\circ\alpha$ and $\bar\beta\circ\bar\alpha$
are projectors. Observe that $\operatorname{Ker}\beta\circ\alpha$
is generated by $\mathcal E(\lambda)$ for all $f_\lambda=(<\ldots)$
or  $(\times\ldots)$.  $\operatorname{Ker}\bar\beta\circ\bar\alpha$
is generated by $[P_\lambda]$ with $f_\lambda=({}^<_{\times^k}\ldots)$
and $[P_{\lambda_+}]-[P_{\lambda_-}]$ with $f_\lambda=((\pm)\times^k\ldots)$. 
By Proposition \ref{prop1} if  $f_\lambda=({}^<_{\times^k}\ldots)$,
then $[P_\lambda]=\sum a(\lambda,\mu)\mathcal E(\mu)$ with $f_\mu=(<\ldots)$
and  if  $f_\lambda=((\pm)\times^k\ldots)$, then $[P_{\lambda_+}]-[P_{\lambda_-}]=\sum a(\lambda,\mu)\mathcal E(\mu)$
with $f_\mu=(\times\ldots)$. That implies
$$\operatorname{Ker}\bar\beta\circ\bar\alpha=\operatorname{Ker}\beta\circ\alpha\cap{}_\mathbb Q\mathcal K_{odd}(P).$$ 
Similarly
$$\operatorname{Im}\bar\beta\circ\bar\alpha=\operatorname{Im}\beta\circ\alpha\cap{}_\mathbb Q\mathcal K_{odd}(P).$$ 
The statement follows.
\end{pf}

\begin{leme}\label{rel4} One has 
$$\beta|_{{}_\mathbb Q\mathcal K_{odd}(P)}=\bar\beta,\,\,\alpha|_{{}_\mathbb Q\mathcal K_{odd}(P)}=\bar\alpha.$$
\end{leme}
\begin{pf}
 By definition $\beta([P])=\bar\beta([P])$ if $P$ is
typical or a direct sum of two typical PIMs.
By Lemma \ref{uniqueev} it is sufficient to check that
$\beta([P])=\bar\beta([P])$ implies $\beta(T[P])=\bar\beta(T[P])$ if
$T$ does not decrease the degree of atypicality of $P$. Indeed we have
$$\bar\beta(T([P]))=\bar\beta\circ\bar\alpha(T'(\bar\beta([P])))=\bar\beta\circ\bar\alpha(T'(\beta([P])))$$
and similarly
$$\beta(T([P]))=\beta\circ\alpha(T'(\beta([P]))).$$
Thus the statement about $\beta$ follows from the previous lemma.

The statement about $\alpha$ follows immediately, since 
$$\bar\beta\circ\bar\alpha=\beta\circ\alpha|_{{}_\mathbb Q\mathcal K_{odd}(P)}$$
implies 
$$\alpha|_{{}_\mathbb Q\mathcal K_{odd}(P)}\circ\bar\beta\circ\bar\alpha=\alpha|_{{}_\mathbb Q\mathcal K_{odd}(P)}\circ\beta\circ\alpha|_{{}_\mathbb Q\mathcal K_{odd}(P)},$$
hence 
$\alpha|_{{}_\mathbb Q\mathcal K_{odd}(P)}=\bar\alpha$.

\end{pf}

Now we are ready to prove Theorem \ref{prop1ev}.
The case of a tailless $\lambda$ was covered in Theorem \ref{lem1}. Write
$$[P_\lambda]=\alpha(\beta([P_\lambda]).$$

Use Theorem \ref{prop1}. If $f_\lambda$ has $>$ at the tail we have 
$$\beta([P_\lambda])=[P_{\lambda'}]=\sum a(\lambda',\nu')\mathcal E(\nu'),$$
and since $\alpha(\mathcal E(\nu'))=\mathcal E(\nu)$ we obtain
$$[P_\lambda]=\sum a(\lambda',\nu')\mathcal E(\nu).$$ 

If $f_\lambda$ does not have $>$ at the tail position, then
$$\beta([P_\lambda])=\frac{1}{2}[P_{\lambda'_+}\oplus P_{\lambda'_-}]=\sum a(\lambda',\nu')\mathcal E(\nu').$$
Here all $f_{\nu'}$ have $\circ$ at the tail position. Hence again
$$[P_\lambda]=\sum a(\lambda',\nu')\mathcal E(\nu).$$
\end{pf}

\begin{exs} $G=SOSP(4,2)$.

$\left[P(>><)\right]=\mathcal E(>><),$

$\left[P(>\times)\right]=\mathcal E(>\times\circ)+\mathcal E(>\circ\times),$

$\left[P(><>)\right]=\mathcal E(><>),$

$\left[P(\times\circ >)\right]= \mathcal E((+)\circ\times >)+\mathcal E((-)\circ\times >),$

$\left[P(\times > \circ )\right]= \mathcal E((+)\circ >\times )+\mathcal E((-)\circ > \times ),$

$\left[P({}_\times^>\circ\circ)+P(>\times\circ)\right]= 2\mathcal E(>\circ\times),$

$\left[P({}_\times^>\circ\circ)\right]=-\mathcal E(>\times\circ)+\mathcal  E(>\circ\times).$

$G=SOSP(4,4)$

$\left[P \left(\begin{array}{c} \times \\ \times      \\ \end{array}\right)\right]= - \mathcal E ((+)\circ \; \times
  \; \times) - \mathcal E ((-)\circ \; \times \; \times) + \mathcal E ((+)\circ \; \times \; \circ \; \times) + \mathcal E ((-)\circ \; \times \; \circ \; \times).$

\end{exs}


\begin{thebibliography}{10}


\bibitem{BGG} J. Bernstein, I. Gel'fand, S. Gel'fand, Category of
  $\goth g$-modules. Func. Anal. Appl., 10 (1976), 87--92.

\bibitem{Bott}  R. Bott, Homogeneous vector bundles. Annals of
  Mathematics, 66 (1957), no. 2.

\bibitem{B} J. Brundan,
Kazhdan-Lusztig polynomials and character formulae for the Lie superalgebra $gl(m\vert n)$.  J. Amer. Math. Soc.  16  (2003),  no. 1, 185--231.

\bibitem{Br} R. Brauer, On modular and p-adic representations of
  algebras. Proc. Nat. Acad. Sci. USA, 25. (1939), 252--258.


\bibitem{BS1} J. Brundan, C. Stroppel, Highest weight categories 
arising from Khovanov's diagram algebra I: cellularity.  Preprint, 2008.

\bibitem{BS2} J. Brundan, C. Stroppel, Highest weight categories 
arising from Khovanov's diagram algebra II: Koszulity.  
Transform. Groups 15 (2010), 1-45.

\bibitem{BS3} J. Brundan, C. Stroppel, Highest weight categories 
arising from Khovanov's diagram algebra III: Category $\mathcal{O}$.  
To appear in Represent. Theory (2011). 

\bibitem{BS4} J. Brundan, C. Stroppel,
Highest weight categories arising from Khovanov's diagram algebra IV: 
the general linear supergroup. To appear in J. Eur. Math. Soc. (2011),

\bibitem{CLW}S.-J. Chang, N. Lam, W. Wang, Super duality and irreducible
characters of ortho-symplectic Lie superalgebras. Inventiones
Mathematicae, to appear (2011)
(http://www.springerlink.com/content/y4633362461h8365). 

\bibitem{CPS} E. Cline, B. Parshall, L. Scott, Finite-dimensional algebras and
  highest weight categories. J. Reine Agnew. Math., 391 (1988), 85--99.

\bibitem{VeraCaroI} C. Gruson, V. Serganova, Cohomology of generalized
  supergrassmannians and character formulae for basic classical Lie
  superalgebras. Proc. of the LMS, 101 (2010), no.3, 852-=892.

\bibitem{H} J. Humphreys, Modular representation of classical Lie
  algebras and semisimple groups. J. Algebra, 19 (1971), 51--79.

\bibitem{I} R. Irving, BGG algebra and BGG reciprocity principle.
  J. Algebra, 135 (1990), 363--380. 

\bibitem{J} J. Jantzen,
Representations of algebraic groups. Second edition. Mathematical Surveys and Monographs, 107. American Mathematical Society, Providence, RI, 2003.

\bibitem{Kadv} V. Kac,
Lie superalgebras.  Advances in Math.  26  (1977), no. 1, 8--96. 

\bibitem{Krep} V. Kac,
Characters of typical representations of classical Lie superalgebras.  Comm. Algebra  5  (1977), no. 8, 889--897. 

\bibitem{M} Yu. Manin,
Gauge field theory and complex geometry. Translated from the 1984 Russian original by N. Koblitz and J. R. King. Grundlehren der Mathematischen Wissenschaften [Fundamental Principles of Mathematical Sciences], 289. Springer-Verlag, Berlin, 1997.

\bibitem{MPV} Yu. Manin, I. Penkov, A. Voronov,
Elements of supergeometry. (Russian) Translated in J. Soviet Math. 51 (1990), no. 1, 2069--2083.


\bibitem{MS} I. Musson, V. Serganova,
Combinatorics of character formulas for the Lie superalgebra 
$\mathfrak{gl}(m,n)$. Transformation groups, 16 (2011), no. 2, 555--578.

\bibitem{P} I. Penkov,
Borel-Weil-Bott theory for classical Lie supergroups. (Russian) Translated in J. Soviet Math. 51 (1990), no. 1, 2108--2140.

\bibitem{VSel} V. Serganova,
Kazhdan-Lusztig polynomials and character formula for the Lie
superalgebra 
${gl}(m\vert n)$.  Selecta Math. (N.S.)  2  (1996),  no. 4, 607--651.

\bibitem{SLMN} V. Serganova, On the superdimension of a
  finite-dimensional representation of a basic classical Lie
  superalgebra. Supersymmetry in mathematics and physics. LNM. To appear. 

\bibitem{Sqr} V. Serganova, Quasireductive supergroups. New
  Developments in Lie theory and its applications, Contemprorary
  Mathematics, AMS, 544, 2011, 141--159.

\bibitem{Ssup} V. Serganova,  Structure and representation theory of
  Kac-Moody superalgebras.
  Highlights in Lie algebraic methods, Birkhauser,  2011, to appear.

\bibitem{SV} A. Sergeev, A. Veselov. Grothendieck ring of basic
  classical Lie superalgebras.  Annals of
  Mathematics, 173 (2011), no.2.

\bibitem{Z} Y. M. Zou, Categories of finite-dimensional weight modules
  over type I classical Lie superalgebras. J. of Algebra, 180 (1996),459--482.
\end{thebibliography}
\end{document}